\documentclass[a4paper,10pt]{article}

\usepackage[utf8x]{inputenc}
\usepackage{amsmath}
\usepackage{amssymb}
\usepackage{amsthm}
\usepackage{cite}
\usepackage{pgfplots}
\usepackage{geometry}
\usepackage[authoryear, round]{natbib}
\usepackage{amsfonts}

\usepackage{tikz}
\usetikzlibrary{arrows,matrix,positioning}

\newtheorem{definition}{Definition}

\newtheorem{theorem}{Theorem}

\begin{document}

\title{Stratified Gaussian Graphical Models}
\author{Henrik Nyman$^{1, \ast}$, Johan Pensar$^{1}$, Jukka Corander$^{2}$ \\
$^{1}$Department of Mathematics and Statistics, \AA bo Akademi University, Finland \\
$^{2}$Department of Mathematics and Statistics, University of Helsinki, Finland \\
$^{\ast}$Corresponding author, Email: hennyman@abo.fi}
\date{}

\maketitle

\begin{abstract}
Gaussian graphical models represent the backbone of the statistical toolbox for analyzing continuous multivariate systems. However, due to the intrinsic properties of the multivariate normal distribution, use of this model family may hide certain forms of context-specific independence that are natural to consider from an applied perspective. Such independencies have been earlier introduced to generalize discrete graphical models and Bayesian networks into more flexible model families. Here we adapt the idea of context-specific independence to Gaussian graphical models by introducing a stratification of the Euclidean space such that a conditional independence may hold in certain segments but be absent elsewhere. It is shown that the stratified models define a curved exponential family, which retains considerable tractability for parameter estimation and model selection.
\end{abstract}


\section{Introduction}
Since their original introduction in the 1970's, Gaussian graphical models (GGMs) are by now ubiquitous in statistical analysis of multivariate systems, given their beneficial characteristics regarding modularity and tractability of statistical inference, see \citet{Dempster72}, \citet{Whittaker90}, \citet{Lauritzen96}, \citet{Giudici99}, \citet{Wong03}, \citet{Atay05}, \citet{Jones05}, \citet{Li06}, \citet{Yuan07}, \citet{Carvalho09}, \citet{Sun12}. However, unlike their discrete counterparts, log-linear graphical models, GGMs do not allow for very flexible representation of marginal and conditional dependence between variables, since their characteristics are determined by the properties of the multivariate normal distribution. A particularly attractive generalization of log-linear graphical models and Bayesian networks is to allow the dependence structure to be context-specific, such that an independence between a pair of variables may hold only when their neighbours attain certain values, i.e. a full conditional independence not being present. Such models have been considered both for directed graphs in \citet{Boutilier96}, \citet{Geiger96}, \citet{Chickering97}, and for undirected graphs in \citet{Corander03a}, \citet{Hojsgaard03, Hojsgaard04}. 

Here we adapt the concept of discrete stratified graphical models \citep{Nyman13} to the multivariate Gaussian family by introducing a stratification of the Euclidean space that specifies where context-specific independencies between variables are present and where absent. It is shown that this definition leads to a plausible characterization of the local influence of the neighbours without leading to complex mixture-type models which would be seriously challenging from the inference perspective. As demonstrated earlier by the discrete model families, context-specific independencies lend themselves easily to applications since it is fundamentally natural to consider the dependence between variables to be absent in a given context while being present elsewhere. We establish formally that the stratification leads to a curved exponential family, which retains considerable tractability in terms of parameter estimation and model selection. The remainder of the article is structured as follows. In Section \ref{secHeadSGGM} we introduce stratified Gaussian graphical models (SGGMs) and examine their statistical properties in detail. In Section 3 we consider inference for SGGMs and the last two sections provide illustrations and some concluding remarks, respectively.

\section{Stratified Gaussian graphical models}
\label{secHeadSGGM}
\subsection{Motivating the introduction of stratified Gaussian graphical models}
To provide an informal and intuitive introduction to the core ideas behind SGGMs, we start by considering the classic dataset concerning mathematics marks introduced by \cite{Mardia79}, see also \cite{Whittaker90} and \cite{Edwards00}. For the five variables, listed in Table \ref{mathVar} representing marks of students in different areas of mathematics, \cite{Whittaker90} presented the GGM with the dependence structure defined by the graph in Figure \ref{mathGGM}. \cite{Whittaker90} noted the central role of algebra in this correlation structure, it being connected to all the remaining variables. However, in analogy to discrete models, what if algebra were independent of mechanics provided that the mark for vectors is above a certain threshold, i.e. $X_{1} \perp X_{3} \mid X_{2}>a$? Alternatively, these two variables might also be independent if the third one belongs to a certain interval: $X_{1} \perp X_{3} \mid a<X_{2}<b$. Such context-specific independencies would appear reasonable in many applications, for instance, a certain signal may activate the dependence of other variables only once it reaches high enough a value. In Section 4 we show that the likelihood for the mathematics marks data in fact supports a local dependence structure of the kind hypothesized above.

\begin{table}[h]
\begin{center}
\begin{tabular}{lc}
\hline
Variable & Label \\
\hline
Mechanics & 1 \\
Vectors & 2 \\
Algebra & 3 \\
Analysis & 4 \\
Statistics & 5 \\
\hline
\end{tabular}
\caption{Variables in mathematics mark dataset.}
\label{mathVar}
\end{center}
\end{table}

\begin{figure}[h]
\begin{center}
\includegraphics{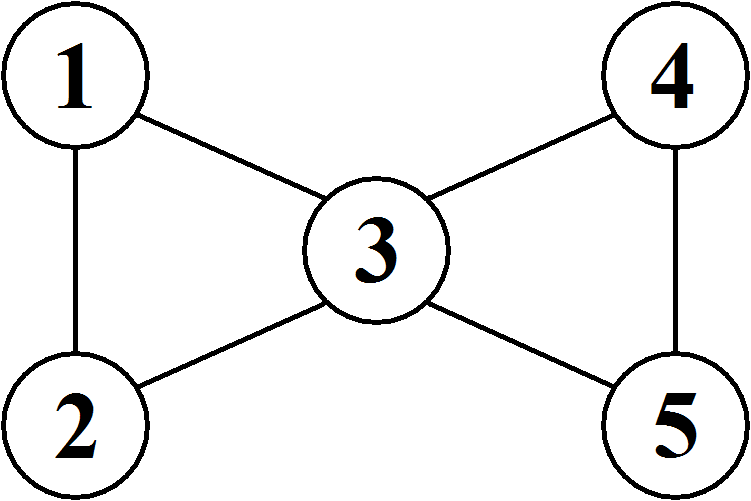}
\end{center}
\caption{Graph representing the dependence structure of the five variables in mathematics mark data.}
\label{mathGGM}
\end{figure}

The standard restrictions on the inverse covariance structure defined by GGMs lead to an exponential family of distributions. Despite of its versatility, such families lack the flexibility to encode context-specific independencies of the type considered above. In contrast, when distinct inverse covariance structures are specified over particular segments of the Euclidean space, one arrives at a more flexible family of distributions which still remains tractable in terms of inference, as shown later in more detail.

\subsection{Notation and preliminaries}
In order to define SGGMs we need the following basic concepts pertaining to Gaussian graphical models. For a more comprehensive treatment, see, for instance, \cite{Lauritzen96} or \cite{Whittaker90}. A $d$-dimensional random vector $X_{\Delta} = (X_{1}, \ldots, X_{d})$ follows a multivariate Gaussian distribution $N(\mu,\Sigma)$ if there exists a mean vector $\mu\in\mathcal{\mathbb{R}}^{d}$ and a positive definite covariance matrix $\Sigma \in \mathcal{\mathbb{R}}^{d\times d}$ such that $X_{\Delta}$ has the probability density function
\[
f_{\mu,\Sigma}(x)=(2\pi)^{-d/2}|K|^{1/2}e^{-1/2(x-\mu)^T K(x-\mu)},
\]
where $K=\Sigma^{-1}$ is the precision matrix of the distribution. This will be denoted $X_{\Delta} \sim N(\mu,\Sigma)$. Independence between two variables or alternatively two sets of variables $X_A$ and $X_B$ can be readily deduced from $\Sigma$, since $\sigma_{\delta,\gamma}=0$ for every $X_{\delta} \in X_A$ and $X_{\gamma} \in X_B$ if and only if the two sets of variables $X_A$ and $X_B$ are marginally independent. Similarly, conditional independence between two variables $X_{\delta}$ and $X_{\gamma}$ can be deduced from the precision matrix, since
\begin{equation}
\label{prec0}
X_{\delta} \perp X_{\gamma} \mid X_{\Delta} \backslash \{X_{\delta}, X_{\gamma}\} \Leftrightarrow k_{\delta, \gamma} = 0.
\end{equation}

Next we define the necessary concepts from graph theory relevant for SGGMs. An undirected graph $G(\Delta, E)$ consists of a set of nodes $\Delta$ and of a set of undirected edges $E\subseteq\{\Delta\times\Delta\}$. For a subset of nodes $A \subseteq \Delta$, $G_{A} = G(A, E_{A})$ is a subgraph of $G$, such that the nodes in $G_{A}$ equal $A$ and the edge set comprises those edges of the original graph for which both nodes are in $A$, i.e. $E_{A} = \{A \times A\} \cap E$. Two nodes $\gamma$ and $\delta$ are adjacent in a graph if $\{\gamma,\delta\}\in E$, that is an edge exists between them. A path in a graph is a sequence of nodes such that two consecutive nodes are adjacent. A cycle is a path that starts and ends with the same node. A chord in a cycle is an edge between two non-consecutive nodes in the cycle. A graph is defined as decomposable if all cycles found in the graph containing four or more unique nodes contain at least one chord. Two sets of nodes $A$ and $B$ are said to be separated by a third set of nodes $S$ if every path between nodes in $A$ and nodes in $B$ contains at least one node in $S$. A graph is defined as complete when all pairs of nodes in the graph are adjacent. A clique in a graph is a set of nodes $C$ such that the subgraph $G_{C}$ is complete and there exists no other set $C^{\ast}$, such that $C \subset C^{\ast}$ and $G_{C^{\ast}}$ is also complete. The set of cliques in a graph are denoted by $\mathcal{C}(G)$. For a decomposable graph the set of separators $\mathcal{S}(G)$ can be obtained through intersections of the cliques of $G$ ordered in terms of a junction tree, see e.g. \cite{Golumbic04}.

Associating each node $\delta \in \Delta$ with a variable $X_{\delta}$, a graphical model is defined by a pair $(G, P_{\Delta})$, where $P_{\Delta}$ is a probability distribution over the variables in $X_{\Delta}$ satisfying a set of restrictions induced by $G$. Given a graphical model it is possible to ascertain if two sets of random variables $X_A$ and $X_B$ are conditionally independent given a third set of variables $X_S$, due to the global Markov property
\[
X_A \perp X_B \mid X_S, \text{ if } S \text{ separates } A \text{ and }B \text{ in } G.
\]
From this property it immediately follows that if there exists no path between the nodes in $A$ and the nodes in $B$ then the two sets of variables $X_A$ and $X_B$ are independent of each other. If the graph $G$ is decomposable $P_{\Delta}$ factorizes as
\begin{equation}
\label{clisep}
P_{\Delta}(X_{\Delta}) = \frac{\prod_{C \in \mathcal{C}(G)}P_C(X_C)}{\prod_{S \in \mathcal{S}(G)}P_S(X_S)}
\end{equation}
\begin{figure}[h]
\begin{center}
\includegraphics{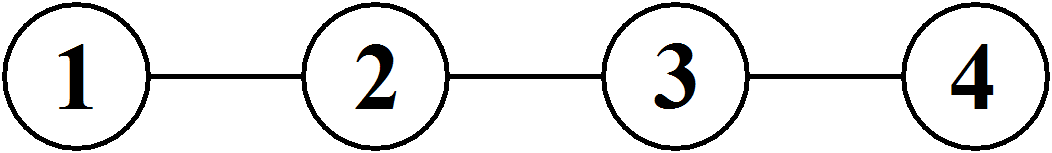}
\end{center}
\caption{Graph containing four nodes and three edges.}
\label{gm}
\end{figure}
If $P_{\Delta}$ is a multivariate Gaussian distribution, $(G, P_{\Delta})$ is a Gaussian graphical model. The restrictions induced by $G$ are manifested as restrictions on the covariance matrix. Consider the graph in Figure \ref{gm} and the conditional independencies of the form described in \eqref{prec0} that it induces. These conditional independencies are
\[
X_{1} \perp X_{3} \mid \{X_{2}, X_{4}\} \qquad X_{1} \perp X_{4} \mid \{X_{2}, X_{3}\} \qquad X_{2} \perp X_{4} \mid \{X_{1}, X_{3}\},
\]
which results in a precision matrix of the form
\[
K=
\begin{pmatrix}
k_{11} & k_{12} & 0 & 0\\
k_{12} & k_{22} & k_{23} & 0\\
0 & k_{23} & k_{33} & k_{34}\\
0 & 0 & k_{34} & k_{44}
\end{pmatrix}.
\]
A statement of conditional independence is such that it holds throughout the outcome space. Consider now the situation with $\Delta=\{1,2,3\}$ where the two variables $X_{2}$ and $X_{3}$ are independent only within a specific segment (or stratum) of the outcome space, for instance, given that the third variable $X_{1}$ is strictly positive, i.e. 
\[
X_{2} \perp X_{3} \mid X_{1}>0.
\]
Such a local independence restriction cannot be captured by standard GGMs, which provides the motivation to develop a class of more general models termed as stratified Gaussian graphical models.

\subsection{Stratified Gaussian graphical models}
\label{secSGGM} SGGMs belong to the class of context-specific models which allow for particular conditional independencies to be present only in a subset, or context, of the outcome space. The appearance and interpretation of an SGGM is quite similar to that of a GGM. For instance, in both types of models an edge between two nodes represents marginal dependence between two variables, conditional dependence is also modeled identically in GGMs and SGGMs. However, for SGGMs a context-specific independence can be introduced by assigning a specific condition to an edge in the graph. For example, in the previous section, where $\Delta=\{1,2,3\}$, the context-specific independence $X_{2} \perp X_{3} \mid X_{1}>0$ can be captured using the graph in Figure \ref{simpleSGGM}. The condition $X_{1}>0$ assigned to the edge $\{2, 3\}$ is referred to as a stratum.
\begin{figure}[h]
\begin{center}
\includegraphics{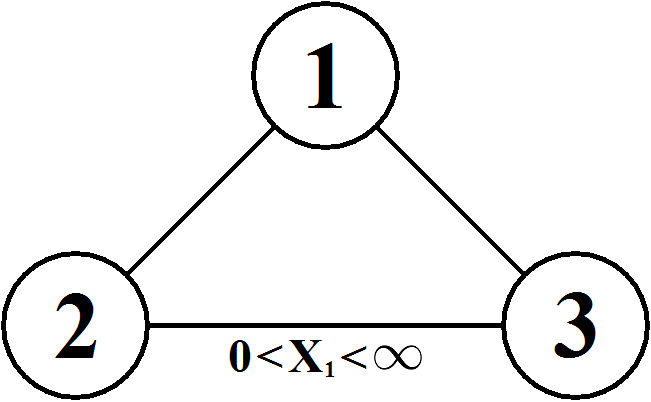}
\end{center}
\caption{Stratified graph over three variables inducing the context-specific independence $X_{2} \perp X_{3} \mid X_{1}>0$.}
\label{simpleSGGM}
\end{figure}
\begin{definition}[\textit{Stratum}]
\label{stratum} Let the pair $(G, P_{\Delta})$ be a graphical model. For all $\{\delta,\gamma\}\in E$, let $L_{\{\delta,\gamma\}}$ denote the set of nodes adjacent to both $\delta$ and $\gamma$. For a non-empty $L_{\{\delta,\gamma\}}$, define the stratum of the edge $\{\delta,\gamma\}$ as the subset $\mathcal{L}_{\{\delta,\gamma\}}$ of the outcome space of the variables in $X_{L_{\{\delta,\gamma\}}}$ such that $X_{\delta} \perp X_{\gamma} \mid X_{L_{\{\delta,\gamma\}}} = x_{L_{\{\delta, \gamma\}}}$, whenever $x_{L_{\{\delta,\gamma\}}} \in \mathcal{L}_{\{\delta, \gamma\}}$. Furthermore, $\mathcal{L}_{\{\delta,\gamma\}}$ must be definable by a union of sets that can be written
\begin{equation}
\{x_{L_{\{\delta,\gamma\}}}:\bigcap_{\zeta \in L_{\{\delta, \gamma\}}}a_{\zeta}<x_{\zeta}<b_{\zeta} \ \text{ for some constants } a_{\zeta}, b_{\zeta} \in \mathbb{R} \cup \{\infty, -\infty\} \}.
\label{stratumSet}
\end{equation}
\end{definition}
Given the definition of a stratum an SGGM is defined as follows. 
\begin{definition}[\textit{Stratified Gaussian graphical model}]
A stratified Gaussian graphical model is defined by the triple $(G, L, P_{\Delta})$, where $G$ is the underlying graph, $L$ is the joint collection of all strata $\mathcal{L}_{\{\delta,\gamma\}}$ for the edges of $G$, and $P_{\Delta}$ is a piecewise Gaussian distribution satisfying the restrictions induced by $(G, L)$. The pair $(G, L)$ is termed a stratified graph (SG). 
\end{definition}
In order to illustrate the difference between densities relating to GGMs and SGGMs consider two covariance matrices
\[
\Sigma^{(1)}=
\begin{pmatrix}
1 & 0.75 & 0.75\\
0.75 & 1 & 0.75\\
0.75 & 0.75 & 1
\end{pmatrix}
\qquad
\Sigma^{(2)}=
\begin{pmatrix}
1 & 0.75 & 0.75\\
0.75 & 1 & 0.5625\\
0.75 & 0.5625 & 1
\end{pmatrix},
\]
where $\Sigma^{(2)}$ is identical to $\Sigma^{(1)}$ except for the elements, $\sigma^{(2)}_{23} = \sigma^{(2)}_{32}$, which are altered to satisfy the condition $k^{(2)}_{23} = k^{(2)}_{32} = 0$. Let $\mu=(0,0,0)^T$ and let $X_{\Delta}$ be a random vector such that $X_{\Delta} \sim N(\mu, \Sigma^{(1)})$. Further, let $Y_{\Delta}$ be a random vector with the density
\[
f_Y(y)=
\begin{cases}
(2\pi)^{-3/2}|K^{(1)}|^{1/2}e^{-1/2(y-\mu)^T K^{(1)}(y-\mu)},\text{ if }y_{1} \leq 0,\\
(2\pi)^{-3/2}|K^{(2)}|^{1/2}e^{-1/2(y-\mu)^T K^{(2)}(y-\mu)},\text{ if }y_{1}>0,
\end{cases}
\]
where $K^{(1)}$ and $K^{(2)}$ are the inverse matrices of $\Sigma^{(1)}$ and $\Sigma^{(2)}$, respectively. It is obvious that $f_Y$ defines a proper probability distribution since it is strictly positive and integrates to one due to the fact that using either $\Sigma^{(1)}$ or $\Sigma^{(2)}$ the probabilities $P(Y_1 > 0) = P(Y_1 < 0) = 0.5$. Additionally, this distribution follows the dependence structure induced by the SG in Figure \ref{simpleSGGM}. Figure \ref{visual} illustrates the noticeable difference between the conditional distributions of $X_{1} \mid \{X_{2}, X_{3}\}$ and $Y_{1} \mid \{Y_{2}, Y_{3}\}$, and shows the potential of a stratum to modify the shape of the density function.
\begin{figure}[h]
\begin{center}
\includegraphics[width=0.9\textwidth]{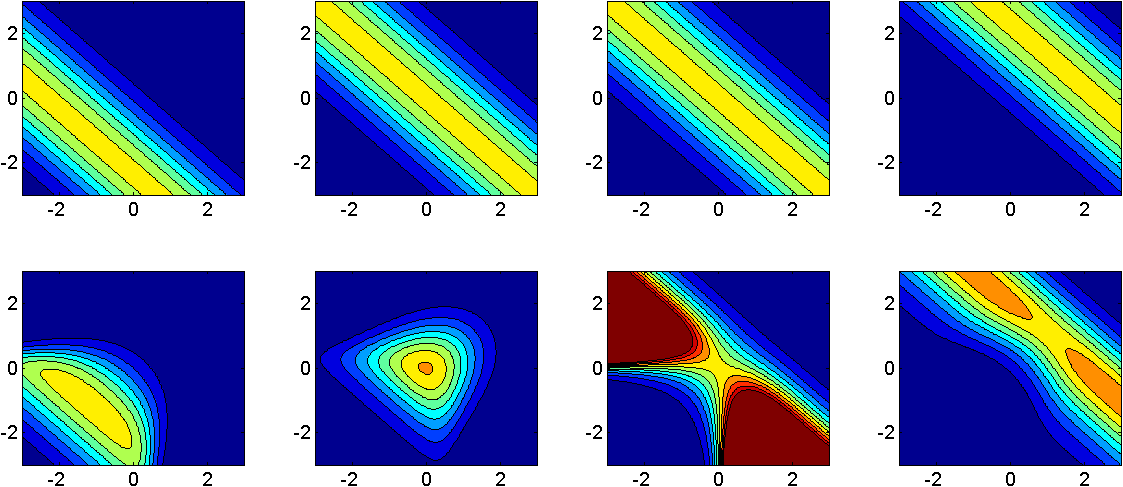}
\end{center}
\caption{Top row: Conditional density function for $X_{1}$ given $X_{2}$ and $X_{3}$ (x- and y-axis), for $x_{1}=-1$, $x_{1}=0$, $x_{1}=0.01$, and $x_{1}=1$. Bottom row: Conditional density function for $Y_{1}$ given $Y_{2}$ and $Y_{3}$ (x- and y-axis), for $y_{1}=-1$, $y_{1}=0$, $y_{1}=0.01$, and $y_{1}=1$.}
\label{visual}
\end{figure}

For an SGGM containing a single stratum the interpretation of the context-specific dependence structure is straightforward. However, for SGGMs containing several strata, the dependence structure can prove less intuitive. To demonstrate this, consider the stratified graph in Figure \ref{combine}a.
\begin{figure}[h]
\begin{center}
\includegraphics{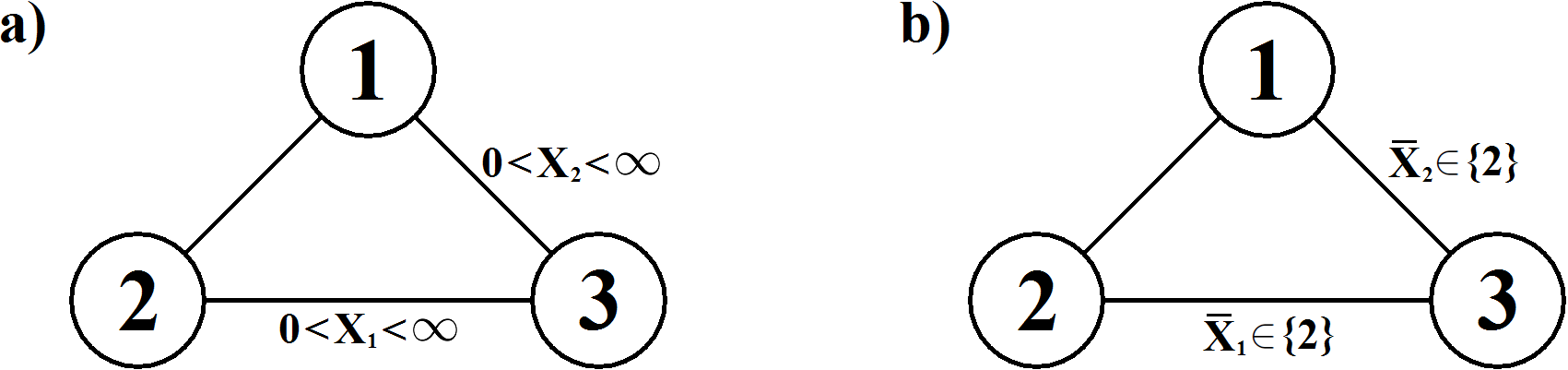}
\end{center}
\caption{In (a) a stratified graph with two strata and in (b) its discretized counterpart.}
\label{combine}
\end{figure}
This SG induces the two separate context-specific independencies: $X_{2} \perp X_{3} \mid X_{1}>0$ and $X_{1} \perp X_{3} \mid X_{2}>0$. Intuitively, this would be interpreted to correspond to the case where the edge $\{1,3\}$ is removed in the context $X_{2}>0$, and the edge $\{2,3\}$ in the context $X_{1}>0$, consequently excluding both edges when the conditions $X_{1}>0$ and $X_{2}>0$ are simultaneously fulfilled. However, a closer scrutiny reveals that this is not a correct interpretation of the induced dependence structure. To show this, let $f(X_{3}  \mid \cdot)$ denote the conditional probability density function of $X_{3}$ given $X_{1}$ and/or $X_{2}$. Assume that $a_{1}$, $a_{2}$ , and $b_{1}$ are all constants $>0$. The following equalities will then hold for a density function following the dependence structure indicated by the SG in Figure \ref{combine}a.
\begin{equation}
\begin{split}
f(X_{3} \mid X_{1} &  =a_{1})=f(X_{3} \mid X_{1}=a_{1},X_{2}=b_{1})=f(X_{3} \mid X_{2}=b_{1})=\\
f(X_{3} \mid X_{1} &  =a_{2},X_{2}=b_{1})=f(X_{3} \mid X_{1}=a_{2}).
\end{split}
\label{larequ}
\end{equation}
The equality $f(X_{3} \mid X_{1}=a_{1})=f(X_{3} \mid X_{1}=a_{2})$, whenever $a_{1}$ and $a_{2}>0$ will be denoted as $X_{3} \perp X_{1} \mid X_{1}>0$ or as $f(X_{3} \mid X_{1}=a_{1})=f(X_{3}|X_{1}>0)$. Analogous calculations will result in the independence restriction $X_{3} \perp X_{2} \mid X_{2}>0$. Recalling the assumption that $a_{1}$ and $b_{1}$ are positive constants, the following observation can be made 
\begin{equation}
f(X_{3} \mid X_{1}=a_{1}, X_{2}=b_{1})=f(X_{3} \mid X_{1}=a_{1}) = f(X_{3} \mid X_{2}=b_{1}),
\label{onetwo}
\end{equation}
i.e. $f(X_{3} \mid X_{1}>0)=f(X_{3} \mid X_{2}>0)$. Now let $a>0$ and $b\leq0$, using the results from \eqref{larequ} and \eqref{onetwo} we obtain the following equalities
\[
f(X_{3} \mid X_{1}=a,X_{2}=b)=f(X_{3} \mid X_{1}=a)=f(X_{3} \mid X_{1}>0)=f(X_{3} \mid X_{2}>0).
\]
A similar result is achieved when $b>0$ and $a\leq0$. Consequently, given that at least one of the conditions $X_{1}>0$ or $X_{2}>0$ is fulfilled, $X_{3}$ is independent of both $X_{1}$ and $X_{2}$. This leads to a model representation with two separate dependence structures, one for the context $X_{1} \leq 0$ and $X_{2} \leq 0$ depicted in Figure \ref{context}a, and another for the context $X_{1}>0$ or $X_{2}>0$ depicted in Figure \ref{context}b.
\begin{figure}[h]
\begin{center}
\includegraphics{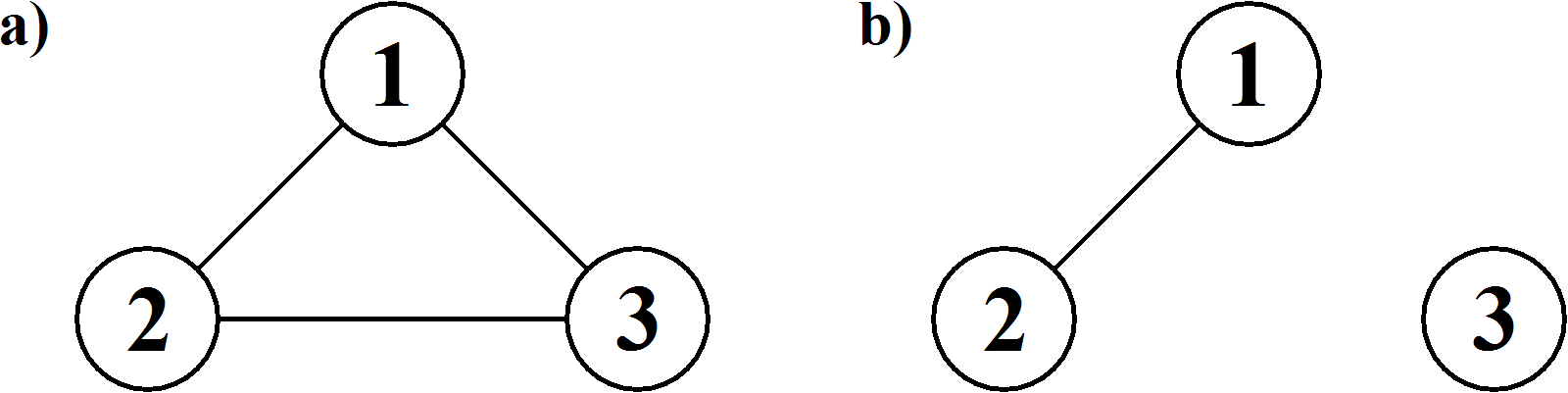}
\end{center}
\caption{Different dependence structures for (a) the context $X_{1} \leq 0$ and $X_{2} \leq 0$ (b) the context $X_{1} > 0$ or $X_{2} > 0$.}
\label{context}
\end{figure}
This simple example illustrates the principle by which the complexity of the dependence structure increases as the number of variables and strata grow. In order to create a coherent modeling framework we therefore introduce the concept of a decomposable stratified graph. This will enable a discretization of the graph, translating a continuous SG into a discrete SG for which the induced dependence structure is thoroughly investigated in \citet{Nyman13}.
\begin{definition} [\textit{Decomposable SG}]
Let $G$ be a decomposable graph and $G_{L} = (G, L)$ an SG with $G$ as its underlying graph. In $G_{L}$, let $E_{L}$ denote the set of all stratified edges (edges associated with a stratum), $E_{C}$ the set of all edges in clique $C \in \mathcal{C}(G)$, and $E_{\mathcal{S}}$ the set of all edges in the separators $\mathcal{S}(G)$ of $G$. The SG is defined as decomposable if
\[
E_{L}\cap E_{\mathcal{S}}=\varnothing,
\]
and
\[
E_{L}\cap E_{C} = \varnothing \hspace{0.4cm} \text{or} \hspace{0cm} \bigcap_{\{\delta,\gamma\}\in E_{L}\cap E_{C}}\hspace{-0.5cm}\{\delta,\gamma\}\hspace{0.1cm}\neq\hspace{0.1cm}\varnothing\hspace{0.1cm}\text{ for all }\hspace{0.1cm}C\in\mathcal{C}(G).
\]
\end{definition}

An SG is defined as decomposable if the underlying graph is decomposable, no strata are associated with edges in any separator, and in every clique all stratified edges have at least one node in common. An SGGM where $(G, L)$ constitutes a decomposable SG is termed a decomposable SGGM. Restricting the underlying graph to be decomposable in combination with not allowing strata to be associated with edges in separators permits a factorization of the density function according to \eqref{clisep}. This is due to the fact that for a decomposable SG the nodes in a stratified edge $\{\delta ,\gamma\}$ and the nodes in $L_{\{\delta,\gamma\}}$ all belong to the same clique. Hence, the strata on an edge in one clique cannot imply changes to the dependence structure between variables associated to nodes in any other clique.

As the separators contain no stratified edges, the dependence structure between variables associated to nodes in a separator is trivial, the same holds for cliques containing no stratified edges. Therefore, further analysis of the dependence structure of decomposable SGGMs can be restricted to cliques of the underlying graph containing one or more stratified edges. This analysis is simplified by first transforming the continuous SG to a discrete SG, for which such an analysis is readily performed, at which point we can revert back to the continuous setting.

We start by defining a discretization process. Let $C$ be the clique under consideration and let $X_{\zeta}$ be a variable such that $\zeta \in C$. Set $\Omega_{\zeta}$ to be the set containing $-\infty$ and $\infty$, along with all the endpoints of the intervals associated to $X_{\zeta}$ when defining the strata on the edges in $C$, i.e.
\[
\Omega_{\zeta} = \{-\infty, \infty\} \ \cup \mathop{\bigcup_{\delta, \gamma \in C}}_{\zeta \in L_{\{\delta, \gamma\}}} \{ \omega : \omega = a_{\zeta} \text{ or } \omega = b_{\zeta} \text{ in a condition } a_{\zeta} < X_{\zeta} < b_{\zeta} \text{ in }  \mathcal{L}_{\{\delta, \gamma\}} \}.
\]
In addition, if an element $\omega$ features as a lower limit in one condition and as an upper limit in another condition it appears twice in $\Omega_{\zeta}$. Next the elements in $\Omega_{\zeta} = (\omega_1, \ldots, \omega_{\text{t}})$ are sorted such that $\omega_i \leq \omega_{i+1}$. The elements in $\Omega_{\zeta}$ are then used as endpoints when defining a set of intervals, $\{(\omega_1, \omega_2), (\omega_2, \omega_3), \ldots, (\omega_{t-1}, \omega_t)\}$. The interval endpoints can either be included or excluded, the first interval, $(\omega_1, \omega_2)$, is always left-open. An interval, $(\omega_{i}, \omega_{i+1})$, is right-open if $\omega_{i} \neq \omega_{i+1}$ and $\omega_{i+1}$ is the upper limit in some stratum condition associated to $\zeta$, i.e $\omega_{i+1} = b_{\zeta}$ in some condition $a_{\zeta} < X_{\zeta} < b_{\zeta}$ in a stratum, otherwise the interval is right-closed. If the interval $(\omega_{i-1}, \omega_{i})$ is right-open the interval $(\omega_{i}, \omega_{i+1})$ is left-closed and correspondingly, if $(\omega_{i-1}, \omega_{i})$ is right-closed the interval $(\omega_{i}, \omega_{i+1})$ is left-open. The last interval, $(\omega_{t-i}, \omega_t)$, is always right-open. Following this method none of the intervals will overlap and the union of all intervals will equal $(-\infty, \infty)$. If a variable $X_{\eta}$ does not appear in any condition in a stratum, or if all conditions are of the form $-\infty < X_{\eta} < \infty$, $\Omega_{\eta}$ will equal $\{-\infty, \infty\}$ resulting in the single interval $(-\infty, \infty)$. Each interval is then assigned an integer value, such that $(\omega_{i}, \omega_{i+1}) \rightarrow i$. This will allow for the translation of a condition in a stratum into a discrete form. If we by $\bar{X}_\zeta$ denote the discrete counterpart of $X_{\zeta}$ the condition $a_{\zeta} < X_{\zeta} < b_{\zeta}$, which is equivalent to $\omega_i < X_{\zeta} < \omega_{j+1}$ for some values $i$ and $j$, can be written as $\bar{X}_{\zeta} \in \Lambda_{\zeta} = \{i, i+1, \ldots, j\}$. Once all the conditions used to define a set according to \eqref{stratumSet} have been converted to there discrete counterparts, the discrete version of the set can be written as 
\[
\bar{X}_{L_{\{\delta, \gamma\}}} \in \mathop{\times}_{\zeta \in L_{\{\delta, \gamma\}}} \Lambda_{\zeta}.
\]
Here $\bar{X}_{L_{\{\delta, \gamma\}}}$ denotes the discretized versions of the variables in $X_{L_{\{\delta, \gamma\}}}$. This means that a single set will be converted into $\prod_{{\zeta} \in L_{\{\delta, \gamma\}}} |\Lambda_{\zeta}|$ discrete outcomes. Transforming all the sets composing a stratum to their discretized versions discretizes the stratum. Once all the strata in an continuous SG have been discretized the result is a discrete SG.

Next we make use of the trait that all stratified edges in a clique of a decomposable SG, and its discrete counterpart, have at least one node in common, this allows us to introduce an ordering of the $d$ variables corresponding to clique $C$ such that the last variable $\bar{X}_{d}$ in the ordering corresponds to the node found in all stratified edges. We will define the variables $(\bar{X}_{1},\ldots, \bar{X}_{d-1})$, which are considered pairwise dependent in the entire outcome space, since the edges connecting the corresponding nodes are not stratified, as the parents of $\bar{X}_{d}$ and denote them by $\bar{\Pi}_{d}$. All the changes induced to the dependence structure by the introduction of strata can be seen in the conditional dependence of $\bar{X}_{d}$ given the set of variables $\bar{\Pi}_{d}$.

Context-specific independencies are readily illustrated using conditional probability tables, which assign a specific distribution to $\bar{X}_{d}$ for each outcome of $\bar{\Pi}_{d}$. Instead of $\bar{X}_{d}$ being assigned a unique distribution for each outcome of $\bar{\Pi}_{d}$, a partition of the outcome space of $\bar{\Pi}_{d}$ is devised, such that any two outcomes in the same block induce the same distribution of $\bar{X}_{d}$. Given our discretized SG we can utilize conditional probability tables, as a discrete stratum conveniently merges parent outcomes creating a partition of the outcome space of $\bar{\Pi}_{d}$. 

Consider the edge $\{(d - 1), d\}$ with the associated discrete stratum $(\bar{X}_1 = \lambda_1, \ldots, \bar{X}_{d-2} = \lambda_{d-2})$. This induces the context-specific independence $\bar{X}_{d-1} \perp \bar{X}_{d} \mid \{\bar{X}_1 = \lambda_1, \ldots, \bar{X}_{d-2} = \lambda_{d-2}\}$, which in terms of a conditional probability table corresponds to merging all outcomes of $\bar{\Pi}_{d}$ where $\bar{X}_1 = \lambda_1, \ldots, \bar{X}_{d-2} = \lambda_{d-2}$. Completing this procedure for all discrete strata will result in the desired partition of the outcome space of $\bar{\Pi}_d$. Each block of the partition is associated with a specific dependence structure which can be ascertained from the conditional probability table. Given a block of outcomes an edge in the SG is deleted if any outcome in the block satisfies any condition in the discrete stratum associated to the edge, any stratum which is not satisfied is also deleted, resulting in each block being associated with a dependence structure determined by an ordinary graph.

Once this entire procedure is completed for all of the cliques found in the underlying graph the results can be combined and translated back to the continuous setting. This yields a set of conditions on the variables $X_{\Delta}$ that form a partition of the outcome space, each condition associated with a specific dependence structure in the form of an ordinary graph. The above described method offers a consistent approach to resolving the dependence structure induced by any decomposable SG, and will later be used when performing inference for SGGMs. We end this section with an example. Consider again the SG in Figure \ref{combine}a, which corresponds to the sets $\Omega_1=\{-\infty, 0, \infty\}$, $\Omega_2=\{-\infty, 0, \infty\}$, and $\Omega_3=\{-\infty, \infty\}$. The resulting intervals along with their discretized values are listed in Table \ref{discrete}.
\begin{table}[htb]
\begin{center}
\begin{tabular}{ccc}
\hline
$X_1$ & $X_2$ & $X_3$ \\
\hline
$(-\infty, 0] \rightarrow 1$ & $(-\infty, 0] \rightarrow 1$ & $(-\infty, \infty) \rightarrow 1$ \\
$(0, \infty) \rightarrow 2$ & $(0, \infty) \rightarrow 2$ & \\
\hline
\end{tabular}
\caption{Discretization of variables in the SG in Figure \ref{combine}a.}
\label{discrete}
\end{center}
\end{table}
\vspace{-0.1cm}
Using these discrete values we can form the discrete SG in Figure \ref{combine}b, from which the conditional probability table for $\bar{X}_3$, found in Table \ref{CPT}, can be derived. The stratum $\bar{X}_1 \in \{2\}$ on the edge $\{2, 3\}$ will merge outcomes $(3)$ and $(4)$, while the stratum $\bar{X}_2 \in \{2\}$ on the edge $\{1, 3\}$ will merge outcomes $(2)$ and $(4)$.
\begin{table}[htb]
\begin{center}
\begin{tabular}{cccc}
\hline
Outcome & $\bar{X}_1$ & $\bar{X}_2$ & Partition \\
\hline
(1) & 1 & 1 & $p_1$ \\
(2) & 1 & 2 & $p_2$ \\
(3) & 2 & 1 & $p_2$ \\
(4) & 2 & 2 & $p_2$ \\
\hline
\end{tabular}
\caption{Conditional probability table resulting from the SGM in Figure \ref{combine}b.}
\label{CPT}
\end{center}
\end{table}
If $\bar{X}_1 \in \{2\}$ or $\bar{X}_2 \in \{2\}$, corresponding to $X_1 \in (0, \infty)$ or $X_2 \in (0, \infty)$, both the edges $\{1, 2\}$ and $\{2, 3\}$ are deleted in that context resulting in the conditions and dependence structures found in Figure \ref{context}.

\subsection{Identifiability of SGGMs}
Identifiability of models with context-specific independence restrictions is of concern since one wishes to avoid situations where two distinct sets of restrictions lead to the same parametric model to retain interpretability and tractability of inference. Also the class of SGGMs necessitates a careful analysis of model identifiability. We illustrate that two distinct decomposable SGs may induce exactly the same dependence structure, as exemplified by the SGs in Figure \ref{same4}a and Figure \ref{same4}c.
\begin{figure}[h]
\begin{center}
\includegraphics[width=0.9\textwidth]{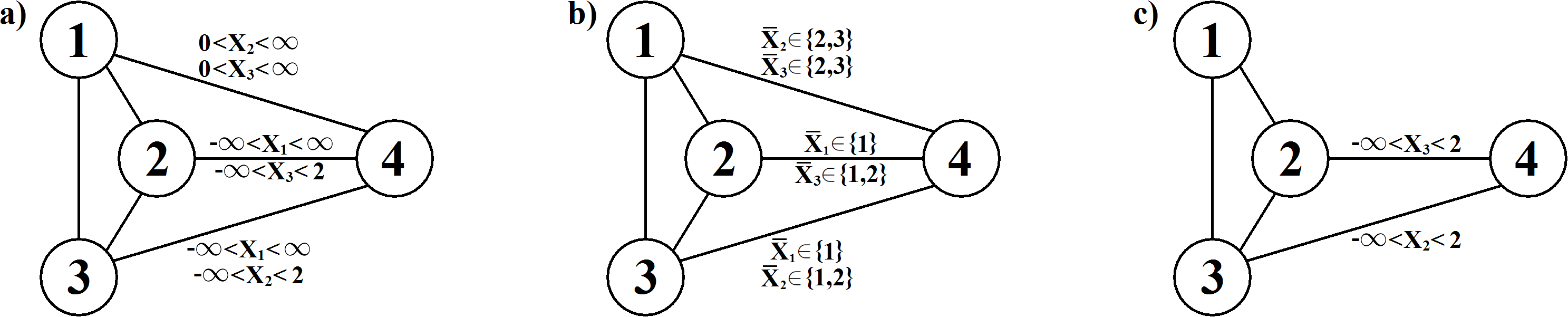}
\end{center}
\caption{Two distinct SGs, (a) and (c), that look completely different but induce the same dependence structure. The graph in (b) is the discretized version of the graph in (a).}
\label{same4}
\end{figure}
It is far from trivial to deduce that these two graphs induce the same dependence structure. To prove that this indeed is the case we apply the method from the previous section to the SG in Figure \ref{same4}a. We start by defining the sets $\Omega_1 = \{-\infty, \infty\}$, $\Omega_2 = \{-\infty, 0, 2, \infty\}$, $\Omega_3 = \Omega_2$, and $\Omega_4 = \Omega_1$, which leads to the discretization in Table \ref{disc4}.
\begin{table}[h]
\begin{center}
\begin{tabular}{cccc}
\hline
$X_1$ & $X_2$ & $X_3$ & $X_4$ \\
\hline
$(-\infty, \infty) \rightarrow 1$ & $(-\infty, 0] \rightarrow 1$ & $(-\infty, 0] \rightarrow 1$ & $(-\infty, \infty) \rightarrow 1$ \\
 & $(0, 2) \rightarrow 2$ & $(0, 2) \rightarrow 2$ &  \\
 & $[2, \infty) \rightarrow 3$ & $[2, \infty) \rightarrow 3$ & \\
\hline
\end{tabular}
\caption{Discretization of variables in the SG in Figure \ref{same4}a.}
\label{disc4}
\end{center}
\end{table}
Using this discretization we can form the discrete SG in Figure \ref{same4}b and the corresponding conditional probability table for $\bar{X}_4$ in Table \ref{CPT4}. The stratum on the edge $\{3, 4\}$ merges outcomes $(1)$, $(2)$, and $(3)$ and $(4)$, $(5)$, and $(6)$, while the stratum on the edge $\{2, 4\}$ merges the outcomes $(1)$, $(4)$, and $(7)$ and $(2)$, $(5)$, and $(8)$, resulting in a partition with only two blocks. Both blocks contain outcomes that satisfies the stratum on the edge $\{1, 4\}$, meaning that it will in all cases be deleted and can therefore be deleted from the underlying graph as well. Deleting the edge will result in $X_1$ being removed from $L_{\{2, 4\}}$ and $L_{\{3, 4\}}$. This, however, is irrelevant as the condition placed on $X_1$ in the strata $\mathcal{L}_{\{2, 4\}}$ and $\mathcal{L}_{\{3, 4\}}$ is of the form $-\infty < X_1 < \infty$ and can therefore be removed from the strata. As we can now see the SG in Figure \ref{same4}a can be transformed into the SG in Figure \ref{same4}c without altering the dependence structure.
\begin{table}[h]
\begin{center}
\begin{tabular}{ccccc}
\hline
Outcome & $\bar{X}_1$ & $\bar{X}_2$ & $\bar{X}_3$ & Partition \\
\hline
(1) & 1 & 1 & 1 & $p_1$ \\
(2) & 1 & 1 & 2 & $p_1$ \\
(3) & 1 & 1 & 3 & $p_1$ \\
(4) & 1 & 2 & 1 & $p_1$ \\
(5) & 1 & 2 & 2 & $p_1$ \\
(6) & 1 & 2 & 3 & $p_1$ \\
(7) & 1 & 3 & 1 & $p_1$ \\
(8) & 1 & 3 & 2 & $p_1$ \\
(9) & 1 & 3 & 3 & $p_2$ \\
\hline
\end{tabular}
\caption{Conditional probability table for $\bar{X}_4$ resulting from the discrete SG in Figure \ref{same4}b.}
\label{CPT4}
\end{center}
\end{table}
In order to construct a class of SGs in which no two graphs induce the same dependence structure, the concept of maximal regular SG is introduced.
\begin{definition}[\textit{Maximal regular SG}]
A decomposable SG is defined as maximal regular if for no edge $\{\delta, \gamma\}$ in $G$ does the set $\mathcal{L}_{\{\delta, \gamma\}}$ encompass the entire outcome space of the variables $X_{L_{\{\delta, \gamma\}}}$, nor can the set $\mathcal{L}_{\{\delta, \gamma\}}$ be expanded without altering the dependence structure.
\end{definition}
An SGGM where $(G, L)$ constitutes a maximal regular SG is termed a maximal regular SGGM.

\begin{theorem}
Two maximal regular SGs induce the same dependence structure if and only if they are identical.
\label{theoremMR}
\end{theorem}

\begin{proof}[Proof of Theorem~\ref{theoremMR}]
Assume that $G_L^1 = (G^1,L^1)$ and $G_L^2=(G^2,L^2)$ are two distinct maximal regular SGs that induce the same dependence structure. Further, assume that the underlying graphs $G^1$ and $G^2$ differ in at least one edge $\{\delta, \gamma\}$ which is present in $G^1$ but not in $G^2$. For $G_L^2$ it then holds that $X_{\delta} \perp X_{\gamma} \mid X_{\Delta} \backslash (X_{\delta}, X_{\gamma})$. For the same to be true for $G_L^1$ the stratum $\mathcal{L}^1_{\{\delta, \gamma\}}$ has to consist of the entire outcome space of the variables $X_{L_{\{\delta, \gamma\}}}$. This contradicts the assumption that $G_L^1$ is maximal regular. Therefore, it can be concluded that $G^1 = G^2$. In order to prove that $L^1$ and $L^2$ are identical we start by assuming that there exists a set of outcomes $\omega \in \mathcal{L}^2_{\{\delta, \gamma\}}$ such that $\omega \notin \mathcal{L}^1_{\{\delta, \gamma\}}$. This means that the dependence structure for $G_L^1$ and $G_L^2$ encompasses the context-specific independence $X_{\delta} \perp X_{\gamma} \mid X_{L_{\{\delta, \gamma\}}} \in \omega$, and that $\omega$ can be added to $\mathcal{L}^1_{\{\delta, \gamma\}}$ without changing the dependence structure. Again, this contradicts the assumption that $G_L^1$ is maximal regular and proves that $L^1 = L^2$.
\end{proof}

When performing inference restricting the model space to maximal regular SGs will decrease the size of the model space as well as at the same time avoiding the problem of different models being assigned the same likelihood due to fact that they induce identical dependence structures. In the next section we will prove that the family of distributions induced by a decomposable SG is a part of the curved exponential family.

\subsection{SGGMs and curved exponential families}
Distributions in GGMs belong to the exponential family, as shown for instance in \cite{Lauritzen96}. A particularly useful characteristic of the exponential family is the consistency of the model selection criterion introduced by \cite{Schwarz78}, often referred to as the Bayesian information criterion (BIC). \cite{Haughton88} extended the consistency result to the curved exponential family, which we will utilize for model selection among SGGMs. For an introduction to the statistical theory for the exponential and curved exponential family, see, for instance, \cite{DasGupta11}.

The probability density function of a distribution in a decomposable SGGM is a function depending on the parameters corresponding to those of a multivariate normal distribution, i.e. the covariance $\Sigma$ and mean $\mu$. However, as is in general assumed for GGMs, we restrict the mean to zero and ignore it in the remainder of the article. In section \ref{secSGGM} it was established that an SG induces a set of conditions on the variables $X_{\Delta}$, that result in a partition of the outcome space where each block is associated with a distinct dependence structure represented by an ordinary graph. These conditions, which are functions of $x$, are denoted as $c^{(1)},\ldots, c^{(\rho)}$. Using iterative proportional fitting, see, for instance \cite{Whittaker90}, the covariance matrix $\Sigma$ can be manipulated to reflect the dependence structure associated with any of the conditions. As such, each condition $c^{(r)}$ gives rise to a specific covariance matrix $\Sigma^{(r)}$ and corresponding precision matrix $K^{(r)}$ which are completely determined by the covariance matrix $\Sigma$ and the dependence structure associated with $c^{(r)}$. Using this notation, the density function can be written as
\begin{equation}
g_{\Sigma}(x)=\frac{1}{Z}\sum_{r=1}^{\rho}f_{\Sigma^{(r)}}(x)I_{c^{(r)}}(x)=\frac{1}{Z}\sum_{r=1}^{\rho}(2\pi)^{-d/2}|K^{(r)}|^{1/2}e^{-\frac{1}{2}x^T K^{(r)} x}I_{c^{(r)}}(x),\label{SGGMpdf}
\end{equation}
where $f_{\Sigma^{(r)}}(x)$ is the density function of the multivariate normal distribution with covariance matrix $\Sigma^{(r)}$ and $Z$ is a normalizing constant. The terms $I_{c^{(r)}}(x)$ are indicator functions equaling $1$, if $x$ satisfies the condition $c^{(r)}$ and $0$ otherwise. As each term in the sum is constituted by a density function of a multivariate normal distribution multiplied by an indicator function, each term will be strictly positive in the part of the outcome space where the corresponding condition is fulfilled. Since the conditions induce a partition of the entire outcome space, guaranteeing that exactly one of the conditions will be fulfilled for every $x$, it follows that $g_{\Sigma}(x)$ is strictly positive for every $x$. The inclusion of the normalizing constant $Z$ ensures that the integral of $g_{\Sigma}(x)$ over the entire outcome space is equal to one. Determining the value of $Z$ is computationally straightforward as each block of the partition of the outcome space corresponding to a condition $c^{(r)}$ is of the form
\[
\{x:\bigcap_{\zeta}a_{\zeta}^{(r)}<x_{\zeta}<b_{\zeta}^{(r)} \text{, for all $\zeta \in \{1, \ldots, d\}$ for some constants $a_{\zeta}^{(r)}$ and $b_{\zeta}^{(r)}$}\},
\]
which implies that $Z$ can be calculated as
\[
Z=\sum_{r=1}^{\rho}\int_{a_{1}^{(r)}}^{b_{1}^{(r)}}\ldots\int_{a_{d}^{(r)}}^{b_{d}^{(r)}}f_{\Sigma^{(r)}}(x)dx_{d}\ldots dx_{1}.
\]
The following criterion is used to determine whether or not a continuous distribution belongs to the exponential family/curved exponential family.
\begin{definition}[\textit{Exponential family}]
A continuous distribution belongs to the exponential family if the probability density function can be written in the form
\[
f_{\theta}(x)=e^{\sum_{i=1}^{k}\eta_{i}(\theta)T_{i}(x)-\psi(\theta)} h(x),
\]
where $k$, which is the length of vectors $\eta$ and $T$, equals the dimension of the parameter $\theta$. In the case when $k$ exceeds the dimension of $\theta$ the distribution belongs to the curved exponential family.
\end{definition}
\begin{theorem}
The distribution in a decomposable SGGM belongs to the curved exponential family.
\label{theoremEF}
\end{theorem}
\begin{proof}[Proof of Theorem~\ref{theoremEF}]
See Appendix A.
\end{proof}

\section{Inference for SGGMs}
\subsection{Score function for SGGMs}
To perform inference and model selection in the class of SGGMs, we adopt an approximate Bayesian approach based on the model scoring criterion introduced by \cite{Schwarz78} combined with a stochastic search for optimal models. We will apply a non-reversible Markov chain Monte Carlo algorithm, introduced by \cite{Corander06} and further developed in \cite{Corander08}, to identify the model with the optimal score, which consistently approximates the mode of the posterior distribution over the space of the considered models.

Let $\mathbf{X}=(x_{ij})_{i=1,j=1}^{d, n}$ be a matrix consisting of $n$ exchangeable observations of a $d$-dimensional random vector, assuming no missing data. In general, when performing inference in order to ascertain an optimal dependence structure the score function would equal the posterior probability, defined as
\[
P(G_{L} \mid \mathbf{X})=\frac{P(G_{L},\mathbf{X})}{P(\mathbf{X})}=\frac{P(\mathbf{X} \mid G_{L})P(G_{L})}{\sum_{G_{L}\in\mathcal{G}}P(\mathbf{X} \mid G_{L})P(G_{L})},
\]
where $P(G_{L})$ is a prior distribution on the space of SGs, denoted by $\mathcal{G}$, and $P(\mathbf{X} \mid G_{L})$ is the marginal likelihood calculated as the expectation $\int_{\Theta_{G_{L}}}P(\mathbf{X} \mid \theta_{G_{L}},G_{L}) P(\theta_{G_{L}}) d \theta_{G_{L}}$ of the likelihood with respect to the prior distribution of the parameters of $G_{L}$. Since analytical calculation of the marginal likelihood appears intractable for stratified graphs, in contrast to ordinary graphs under conjugate priors \citep{Dawid93}, we use a consistent approximation of the log marginal likelihood based on the maximum likelihood function under the restrictions imposed by $G_{L}$ combined with the BIC penalty function:
\begin{equation}
\log P(\mathbf{X} \mid G_{L})\approx\log l(\mathbf{X} \mid G_{L})-\frac{k(G_{L})}{2}\log n=S(G_{L} \mid \mathbf{X}),
\label{score}
\end{equation}
here $l(\mathbf{X} \mid G_{L})$ is the maximized value of the likelihood function induced by $G_{L}$ and $k(G_{L})$ is the cardinality of the parameter space induced by $G_{L}$. The likelihood function can be expressed in the form
\[
\prod_{j=1}^{n}g_{\Sigma}(x_{j})=\prod_{j=1}^{n}\frac{1}{Z}\sum_{r=1}^{\rho}f_{\Sigma^{(r)}}(x_{j})I_{c^{(r)}}(x_{j}),
\]
here $x_{j}$ is the $j$th column of $\mathbf{X}$. Using the Bayesian information criterion approximation, the problem reduces to finding the maximum likelihood estimate of the model parameters for any given candidate of the dependence structure. For ordinary decomposable graphs the maximum likelihood estimate $\hat{\Sigma}$ of $\Sigma$ is analytically tractable, and it is relatively simple to obtain an estimate even for non-decomposable graphs. A straightforward approach is to first calculate the maximum likelihood estimate of the covariance without imposing any constraints and then enforce the constraints of the graph by using iterative proportional fitting. However, this method is not directly applicable to SGs, since the dataset is partitioned, with each partition associated with its own dependence structure.

Currently we are not aware of any method for analytically calculating $\hat{\Sigma}$ for SGs. Instead, we exploit a method that cyclically
optimizes one element $\sigma_{ij}$ at a time until sufficient numerical convergence is reached for the whole structure $\Sigma$. Importantly, not all elements in the covariance structure need to be optimized in this manner, since they will in each instance $\Sigma^{(1)},\ldots,\Sigma^{(\rho)}$ be determined by certain other elements in $\Sigma$. To illustrate this, consider the example with the SG in Figure \ref{same4}c for which the corresponding covariance matrix is shown below.

\begin{center}
\begin{tikzpicture}
\matrix [matrix of math nodes,left delimiter=(,right delimiter=)] (m)
{
\sigma_{11} & \sigma_{12} & \sigma_{13} & \sigma_{14} \\
\sigma_{12} & \sigma_{22} & \sigma_{23} & \sigma_{24} \\
\sigma_{13} & \sigma_{23} & \sigma_{33} & \sigma_{34} \\
\sigma_{14} & \sigma_{24} & \sigma_{34} & \sigma_{44} \\
};
\draw[color=red] (m-3-1.south west) -- (m-3-2.south west) -- (m-4-2.south west) -- (m-4-1.south west) -- (m-3-1.south west);
\draw[color=red] (m-1-3.north east) -- (m-1-4.north east) -- (m-2-4.north east) -- (m-2-3.north east) -- (m-1-3.north east);
\draw[color=blue] (m-3-3.south east) -- (m-3-4.south east) -- (m-4-4.south east) -- (m-4-3.south east) -- (m-3-3.south east);
\draw[color=blue] (m-1-1.north west) -- (m-1-4.north west) -- (m-4-4.north west) -- (m-4-1.north west) -- (m-1-1.north west);
\draw[color=green] (m-1-3.south east) -- (m-1-4.south east) -- (m-4-4.north east) -- (m-4-3.north east) -- (m-1-3.south east);
\draw[color=green] (m-3-1.south east) -- (m-3-4.south west) -- (m-4-4.south west) -- (m-4-1.south east) -- (m-3-1.south east);
\end{tikzpicture}
\end{center}

The elements inside the blue rectangles are those that will be identical for all covariance matrices $\Sigma,\Sigma^{(1)},\ldots,\Sigma^{(\rho)}$. This follows from the fact that iterative proportional fitting never changes elements on the diagonal, i.e. $\sigma_{ii}$, nor does it change an element $\sigma_{ij}$ if there is an edge between nodes $i$ and $j$ in the graph. Since the elements inside the blue rectangles correspond to pairs of nodes connected by edges with no associated strata, these elements will be identical in $\Sigma^{(1)},\ldots,\Sigma^{(\rho)}$.

The elements inside the green rectangles corresponds to pairs of nodes connected by stratified edges, meaning that these elements will be changed in some of the covariance matrices $\Sigma^{(1)},\ldots,\Sigma^{(\rho)}$. The elements in the red rectangles correspond to pairs of nodes that are not connected by an edge, which means that they can take different values in all of the covariance matrices. The value of these elements in $\Sigma^{(1)},\ldots,\Sigma^{(\rho)}$ will be completely determined by the other elements and it is therefore unnecessary to include them in the optimization process. The numerical optimization of each element $\sigma_{ij}$ is carried out as follows.

\begin{enumerate}
\item Choose the value of two scalars $\epsilon< \delta$.

\item Evaluate the likelihood function using both values $\sigma_{ij}\pm\epsilon/2$ as candidates of $\sigma_{ij}$. If $\sigma_{ij} + \epsilon/2$ yields a higher likelihood than $\sigma_{ij}$, set $d=1$, otherwise if $\sigma_{ij}-\epsilon/2$ yields higher likelihood, set $d=-1$. If both $\sigma_{ij}\pm\epsilon/2$ produce inferior results or results in $\Sigma$ being negative semi-definite, stop the optimization process for $\sigma_{ij}$.

\item Evaluate the likelihood function using $\sigma_{ij}^{\ast}=\sigma_{ij} + d \times \delta$ as a candidate of $\sigma_{ij}$. If this improves the likelihood function, set $\sigma_{ij}=\sigma_{ij}^{\ast}$, otherwise set $\delta=\max(\epsilon/2,\delta/2)$ and repeat step 3.

\item Repeat steps 2 - 4.
\end{enumerate}
It was determined above that some covariance elements will be completely determined by the other elements and need not be included in the optimization process. These elements will, however, affect whether or not $\Sigma$ is positive definite or not. Therefore, it will be necessary to, at the end of each cycle, transform $\Sigma$ to comply with the restrictions induced by the underlying graph.

Additionally, a criterion to determine whether or not $\Sigma$ has converged needs to be defined. One possible approach would be to terminate the procedure when a whole cycle has been completed without changing any of the elements in $\Sigma$. Another, more pragmatic definition which we apply in our illustrations, depends on the improvement in the likelihood function after each cycle. When the resulting improvement in the likelihood function during a complete cycle is less than a predefined tolerance value, the estimation procedure is terminated. Given an appropriate starting value, for example, the sample covariance matrix, the time needed for convergence is tractable for a moderate number of variables.

The parameter space for multivariate normal distributions corresponding to the complete graph spanning $d$ variables contains $(d^{2}+d)/2+d$ free parameters, equaling the sum of parameters found in $\Sigma$ and $\mu$. Although we have restricted $\mu$ to zero, this corresponds to a preprocessing of the data and therefore the contribution of $\mu$ to the number of free parameters is included. Removing an edge from the graph corresponds to removing a free parameter from the parameter space. This can be seen from the precision matrix as each absent edge in the graph corresponds to a conditional independence of the type found in \eqref{prec0}, which in turn corresponds to forcing an element to equal zero in the precision matrix. Using this method we can deduce the number of free parameters induced by the underlying graph of the SG, however, we also need to take into account the number of parameters needed to define the set of strata included in the SG. The number of these parameters can be specified by studying the structure of a stratum. Each set of the form in \eqref{stratumSet} requires the introduction of $2 \times |L_{\{\delta,\gamma\}}|$ new parameters, where $|L_{\{\delta,\gamma\}}|$ denotes the number of nodes adjacent to both $\delta$ and $\gamma$. Let $|\mathcal{L}_{\{\delta, \gamma\}}|$ denote the number of sets of form \eqref{stratumSet} used to define $\mathcal{L}_{\{\delta, \gamma\}}$ and let $|E^{c}|$ denote the difference between the number of edges in $G$ and the corresponding complete graph. The cardinality of the parameter space induced by an SG then equals
\[
k(G_{L})=\frac{d^{2}+d}{2}+d-|E^{c}|+\sum_{\delta=1}^{d-1}\sum_{\gamma=\delta+1}^{d}|\mathcal{L}_{\{\delta,\gamma\}}| \times 2 \times |L_{\{\delta,\gamma\}}|.
\]
As an example, for the SG in Figure \ref{same4}c the cardinality of the parameter space equals
\[
k(G_{L})=(4^{2}+4)/2+4-1+1 \times 2 \times 1+1 \times 2 \times 1=17.
\]
Given the obtained maximum likelihood estimate $\hat{\Sigma}$ and the cardinality of the parameter space induced by the SG, equation \eqref{score} can be used to approximate the posterior probability for SGGMs.

\subsection{Non-reversible Markov chain Monte Carlo search for SGGMs}
The learning algorithm described below belongs to the class of non-reversible Metropolis-Hastings algorithms, introduced by \cite{Corander06} and later further generalized and applied to learning of graphical models in \cite{Corander08}. Let $\mathcal{M}$ denote the finite space of models over which the aim is to identify the model with the optimal score. For $M\in\mathcal{M}$, let $Q(\cdot \mid M)$ denote the proposal function used to generate a new candidate model given any model $M$. Under the generic conditions stated in \cite{Corander08}, the probability assigned to any particular candidate by $Q(\cdot \mid M)$ need not be explicitly calculated or known, as long as it remains unchanged over the iterations of the algorithm and the resulting chain satisfies the condition that all states can be reached
from any other state in a finite number of steps. Assume that the model learning is initialized by a model $M_{0}$. At iteration $t=1,2,...$ of the algorithm, $Q(\cdot \mid M_{t-1})$ is used to generate a candidate model $M^{\ast}$, which is accepted with the probability
\[
\min\left(  1,\frac{P(M^{\ast})P(\mathbf{X} \mid M^{\ast})}{P(M_{t-1})P(\mathbf{X} \mid M_{t-1})}\right)  \label{accept}
\]
where $P(M)$ is a prior probability assigned to $M$ and $P(\mathbf{X} \mid M)$ is the marginal likelihood of the dataset $\mathbf{X}$ given $M$.

Contrary to reversible Markov Chains, for non-reversible Markov chains the posterior probability $P(M \mid \mathbf{X})$ is not approximated by the stationary distribution. Instead, a consistent approximation of $P(M \mid \mathbf{X})$ is obtained by considering the space of distinct models $\mathcal{M}_{t}$ visited by time $t$ such that
\[
\hat{P}(M \mid \mathbf{X})=\frac{P(\mathbf{X} \mid M)P(M)}{\sum_{M\in\mathcal{M}_{t}}P(\mathbf{X} \mid M)P(M)}.
\]
\cite{Corander08} proved under rather weak conditions that this estimate is consistent, i.e.
\[
\hat{P}(M \mid \mathbf{X})\overset{a.s.}{\rightarrow}P(M \mid \mathbf{X}),
\]
as $t\rightarrow\infty$. Since our main interest lies in finding the posterior optimal model, i.e.
\[
\arg\mathop{\max}_{M\in\mathcal{M}}P(M \mid \mathbf{X}).
\]
it will suffice to identify
\[
\arg\mathop{\max}_{M\in\mathcal{M}}P(\mathbf{X} \mid M)P(M).
\]
Throughout this article we use a uniform prior distribution over the model space, which further simplifies the search algorithm as the prior then cancels out in all the formulas. As an approximation, we will replace the marginal likelihood with the score function $S(G_{L} \mid \mathbf{X})$, resulting in a stochastic search for the model with optimal score. The proposal function used in our algorithm is available in Appendix B.

As the model space consisting of all maximal regular SGs grows extremely fast in relation to the number of nodes in the system, identifying a good initial state $M_{0}$ for the non-reversible Markov chain is of importance. A viable initial state is found by first conducting a search for the optimal undirected graph and then performing a search for strata separately for each edge included in this graph. Combining the graph with the resulting strata, or a subset of the strata in case the entire set results in a non maximal regular SG, provides a reasonable initial state using less computationally demanding operations compared to an algorithm that can traverse the entire space of maximal regular SGs.

\section{Illustrations}
We start by re-visiting the mathematic marks dataset mentioned in Section \ref{secHeadSGGM}. Conducting a search for the optimal GGM using the framework described in the previous section, but modified to only consider traditional GGMs, results in a model with the graph displayed in Figure \ref{mathGGM}. This model has the score of $-1731.33$. However, by enlarging the model space to also include SGGMs, the optimal model identified has the score $-1730.21$, the corresponding stratified graph is shown in Figure \ref{mathSGGM}.
\begin{figure}[h]
\begin{center}
\includegraphics{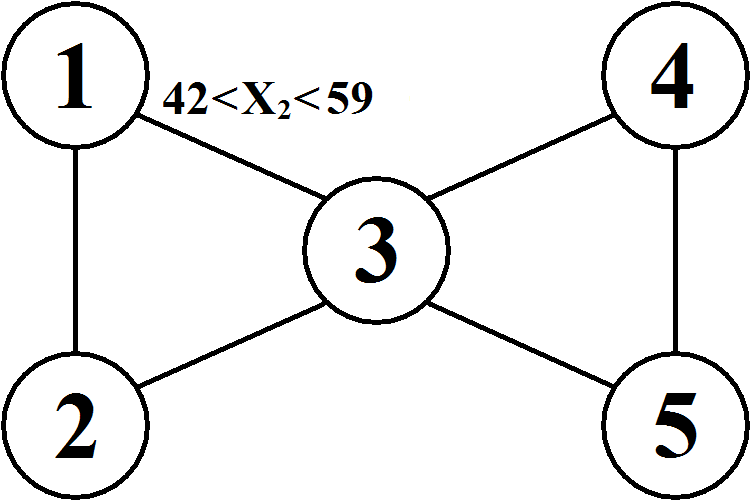}
\end{center}
\caption{Optimal SG for mathematic marks data.}
\label{mathSGGM}
\end{figure}
This SGGM, while having the same conditional dependence structure as the optimal GGM, also incorporates the context-specific independence $X_{1} \perp X_{3} \mid X_{2} \in (42,59)$. The marks for $X_{2}$ range from $9$ to $82$ with the interval $(42,59)$ composing 39 observations. Considering the clique $\{1, 2, 3\}$ along with the entire dataset the partial correlation between variables $X_{1}$ and $X_{3}$ equals $0.3171$. However, if we only consider the subset of data where $X_{2} \in (42,59)$ the corresponding value equals $-0.0017$. Clearly, the included stratum manages to identify a subset of data where $X_{1}$ and $X_{3}$ are, in practice, conditionally independent given $X_2$.

In order to show that our scoring and search methods perform as intended, we generate a synthetic dataset following a distribution where the dependence structure can be represented by the SG in Figure \ref{data8}. The specific distribution is available in Appendix C.
\begin{figure}[h]
\begin{center}
\includegraphics{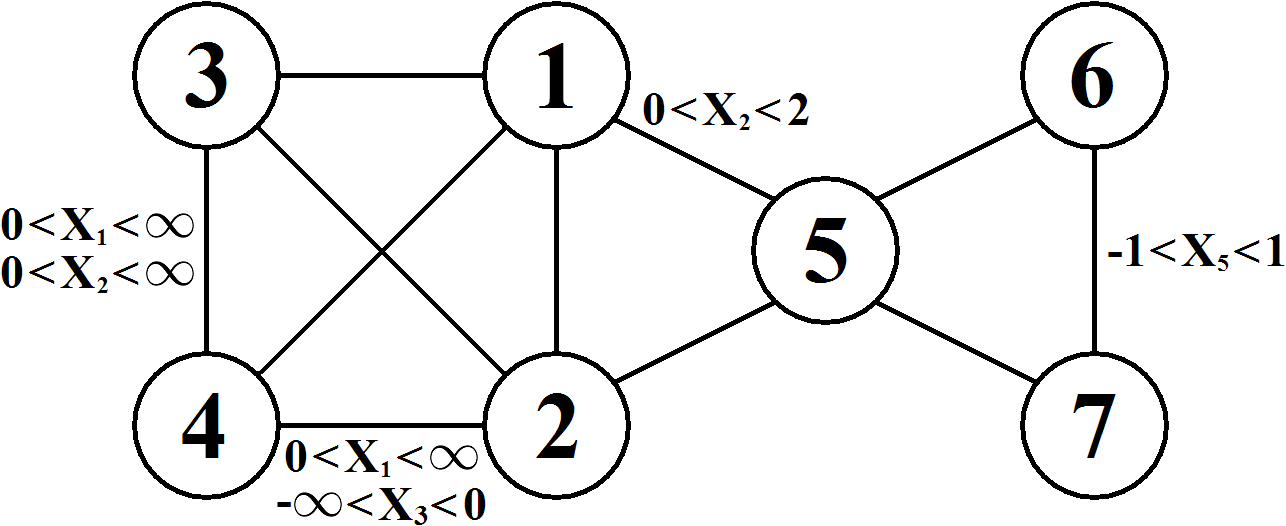}
\end{center}
\caption{Dependence structure of the distribution used to generate synthetic data.}
\label{data8}
\end{figure}
The model space for SGs is in theory infinitely large as the endpoints of a stratum are defined by continuous variables. However, when trying to fit an SG to a dataset, the model space can be considered finite as, for instance, a condition $0< X_{\zeta} <a$ will in practice be the same as $0< X_{\zeta} <b$ if $a\leq b$ and there exist no observations for which $X_{\zeta} \in \lbrack a,b)$. Nevertheless, the model space for SGs is still astronomically large when compared to the model space for ordinary graphs. Consequently, in order to preform solid inference for SGs we generally expect that larger datasets are required than for ordinary graphs.

In our experiments, a dataset containing more than $1,000$ observations would generally yield an inferred model very close to the generating model. The most challenging part is to correctly identify the endpoints for the interacting strata associated to the edges $\{2, 4\}$ and $\{3,4\}$, since minor changes in them only lead to subtle changes in the dependence structure.
\begin{figure}[h]
\begin{center}
\includegraphics[width=0.9\textwidth]{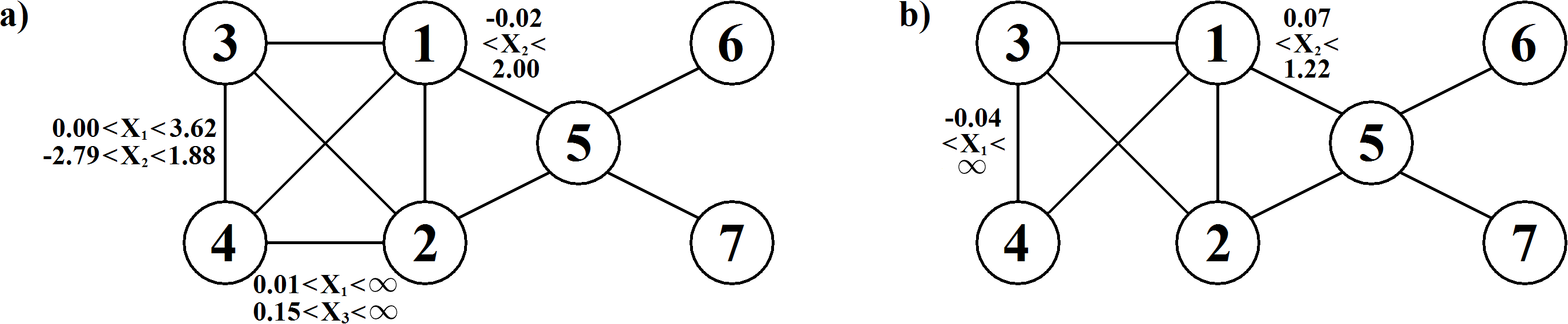}
\end{center}
\caption{SGs with optimal score for datasets containing 1,000 observations (a) and 100 observations (b).}
\label{data8_Ex}
\end{figure}
Several synthetic datasets of $1,000$ and $100$ observations were generated. From these a single representative dataset of each size was chosen to demonstrate how a small number of data points limits the power to infer the generating graph. The dependence structure of the SGGMs with optimal score are displayed in Figure \ref{data8_Ex}. The scores of the generating model for the two datasets are $-6149.06$ and $-674.11$, respectively. The scores for the optimal models are $-6134.03$ and $-656.88$, respectively. Comparing these models with the generating model one can see considerable similarities, however, some discrepancies are also present.

The edge $\{6, 7\}$, which in the generating model is a stratified edge, is missing in both of the optimal models. This can be explained by the amount of observations that belong to the outcome space where the stratum is satisfied. The variance of $X_{5}$, which is the variable defining the stratum, is slightly larger than $1$, which implies that roughly $70\%$ of the observations on $X_{5}$ will reside within the interval $(-1,1)$ satisfying the stratum. This can be compared to the approximately $50\%$ of observations belonging to the stratum associated to the edge $\{1, 5\}$, which is present in both of the optimal models. In other words, for a large majority of observations $X_{6}$ and $X_{7}$ are conditionally independent given $X_{5}$. This in combination with the fact that in the context where $X_{6}$ and $X_{7}$ are conditionally dependent given $X_{5}$, the absolute value of $X_{5}$ is large and will thus have a larger impact on $X_{7}$ than will $X_{6}$. In combination these circumstances will lead to the edge $\{6, 7\}$ being relatively weak, requiring a large amount of data in order to be conclusively supported since the scoring criterion will penalize and attempt to filter out weak associations from the models.

Even for the larger dataset the optimal endpoints in the strata $\mathcal{L}_{\{2,4\}}$ and $\mathcal{L}_{\{3,4\}}$ can differ from those specified in the generating model. However, the actual implications on the dependence structure are not that extreme. For the smaller dataset the edge $\{2, 4\}$ is not present in the optimal model, this can again be explained by the relative weakness of the edge, often resulting in either the edge $\{2, 4\}$ or $\{3, 4\}$ being omitted from the optimal model in the multiple realizations analyzed. The exclusion of one of these edges greatly simplifies the task of finding a suitable stratum for the remaining edge. In summary, even for relatively small datasets, the search method produces a fairly accurate approximation to the generating model. It is, however, evident that in order to capture more subtle dependencies a larger amount of data will be required.

Next we will consider two real datasets involving gene expression data and protein expression data. However, in order to do this it will first be necessary to slightly modify the score function in \eqref{score}. When the number of variables considered in a model is large, Gaussian graphical models have a tendency to be quite dense, often overestimating the number of included edges \citep{Foygel10}. To compensate for this issue \citet{Foygel10} suggested the inclusion of an additional penalty term in the BIC resulting in the \textit{extended BIC}. This corresponds to modifying our score function according to
\begin{equation}
S_{\text{exp}}(G_{L} \mid \mathbf{X}) = \log l(\mathbf{X} \mid G_{L})-\frac{k(G_{L})}{2}\log n -\frac{|E|}{2}\log \kappa |\Delta|,
\label{expScore}
\end{equation}
where $|E|$ denotes the number of edges found in $G_L$, $|\Delta|$ the number of nodes, and $\kappa$ is a tuning parameter. Choosing a suitable value of $\kappa$ will result in a graph with optimal comprehensibility, as it is very difficult to determine strong dependencies from graphs that are very dense and, contrary, graphs with very few edges might leave out fairly strong dependencies. \citet{Foygel10} showed that the extended BIC is asymptotically equivalent to BIC as the number of observations, $n$, goes to infinity. When considering SGGMs operating with fairly sparse graphs facilitates the inference of strata as it is easier to identify viable strata in smaller cliques than in larger cliques. Therefore, the use of the extended BIC can also be extremely useful when inferring SGGMs for large systems.

For the gene expression data 15 variables were randomly chosen from the dataset found in \citet{Hiissa09} containing 335 microarray observations that were quantile normalized prior to our analysis as described in the data source. The considered genes are available in Appendix D. For this dataset the extended BIC score using $\kappa = 2.5$ was used. The resulting optimal SG, shown in Figure \ref{gene}, has a score of $-16967.38$.
\begin{figure}[h]
\begin{center}
\includegraphics[width=0.45\textwidth]{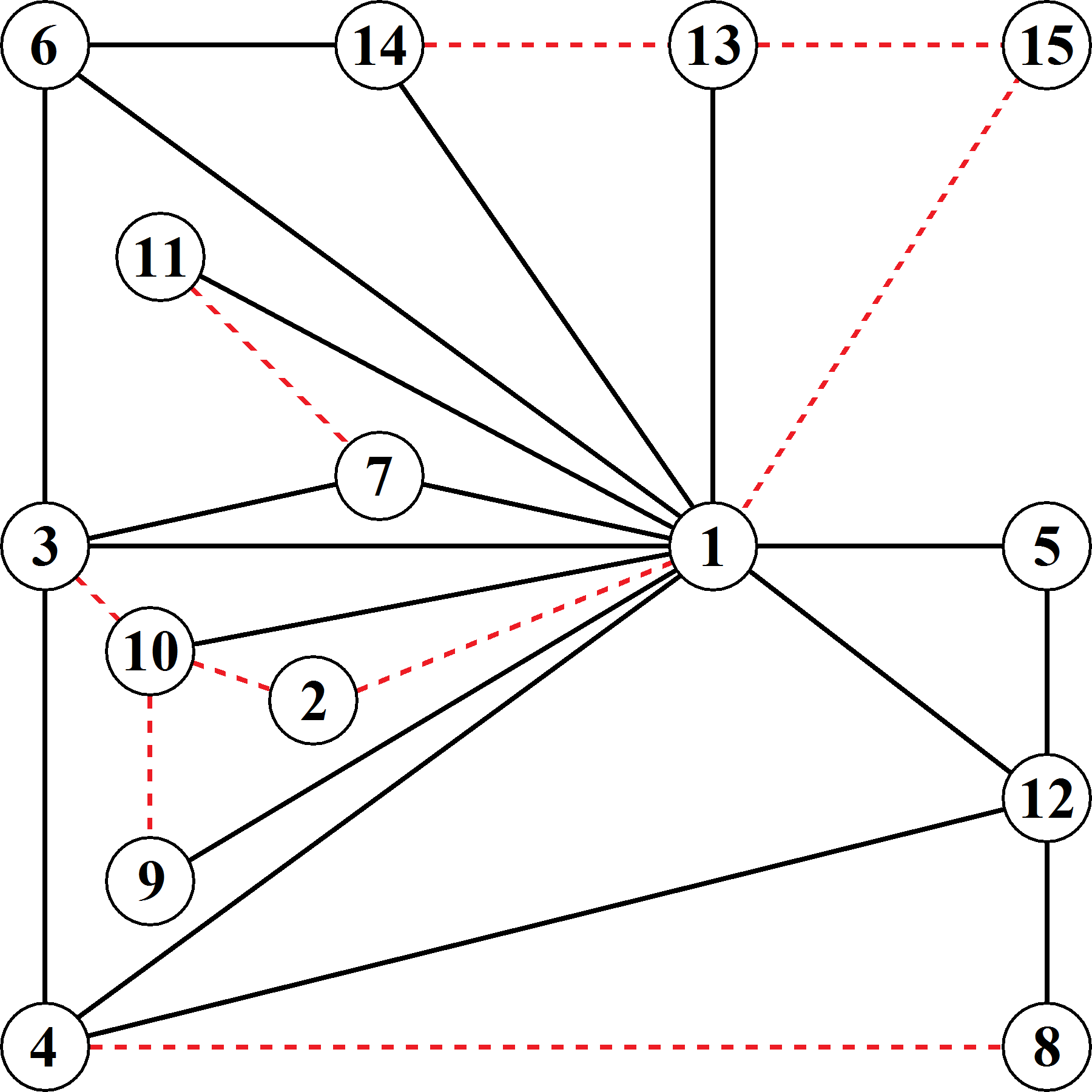}
\end{center}
\caption{Optimal SG for gene expression data.}
\label{gene}
\end{figure}
Instead of giving the strata explicitly the stratified edges are shown using dashed lines. The SG contains 27 edges and nine stratified edges. The underlying graph of the SG is very similar to the graph used in the optimal GGM with the only difference being that the edge $\{2, 10\}$ is included in the SG but not in the ordinary graph. The graph of the optimal GGM has a score of $-17027.92$. If we again consider a single clique, for instance $\{1, 7, 11\}$ containing the context-specific independence $X_{7} \perp X_{11} | X_1 \in (96.37, 118.49)$, the partial correlation between $X_{7}$ and $X_{11}$ is $0.4021$ for the entire dataset. Considering only the data where $X_1 \in (96.37, 118.49)$ the corresponding value is $0.0925$, i.e. considerably closer to 0 compared to the value for the entire data.  

The protein data is taken from \citet{Kornblau09} and contains 256 observations on 51 variables. The optimal SG is displayed using the adjacency matrix in Figure \ref{protein}, the circles represent stratified edges, the triangles edges to which strata could be added while still retaining a decomposable SG, and the squares edges that cannot be stratified in a decomposable SG.
\begin{figure}[h]
\begin{center}
\includegraphics[width=0.45\textwidth]{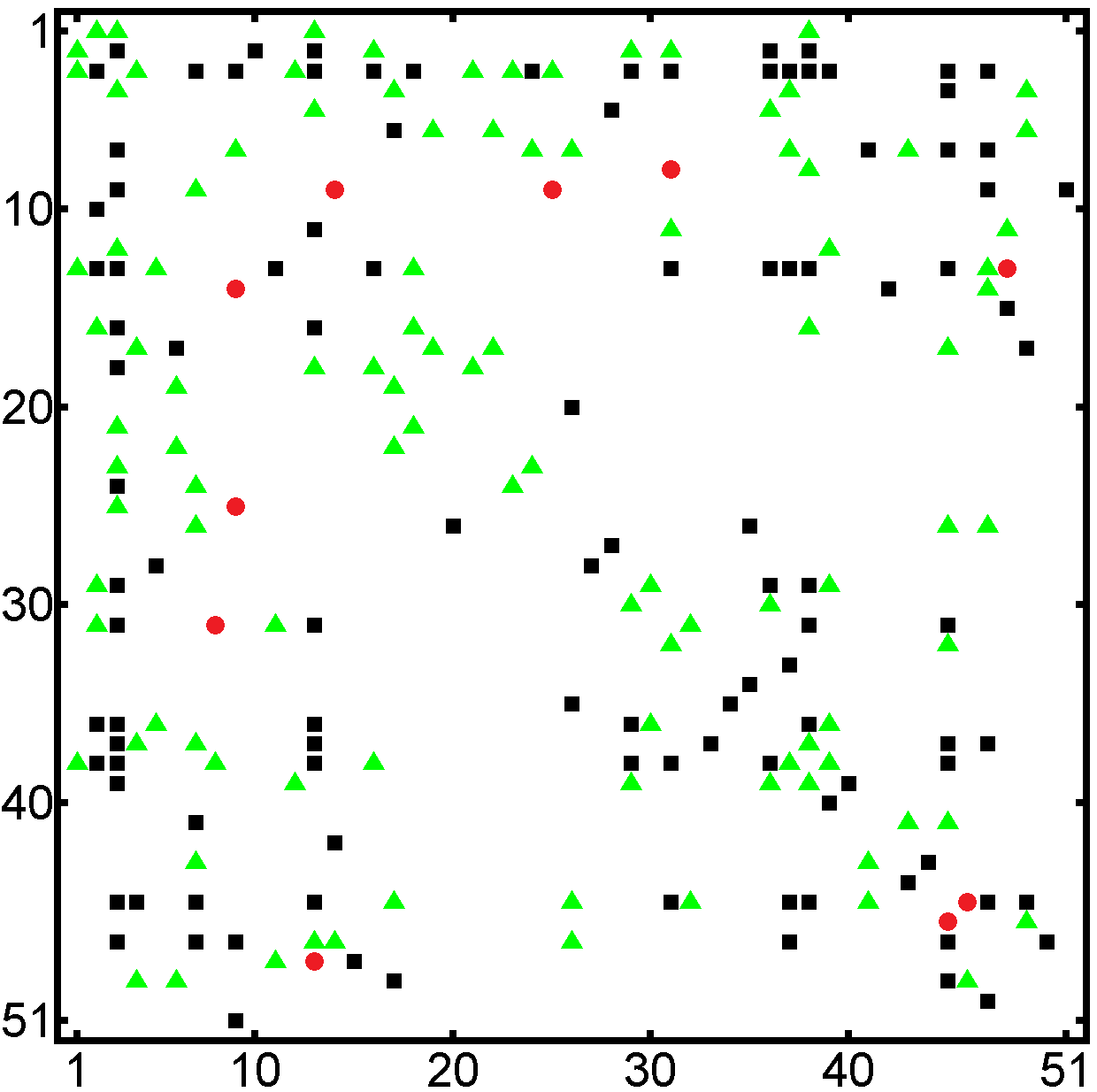}
\end{center}
\caption{Optimal SG for protein expression data displayed using an adjacency matrix.}
\label{protein}
\end{figure}
For this dataset the extended BIC score using $\kappa = 0.5$ was used. The underlying graph of the inferred SG coincides with the graph of the optimal GGM. However, the SG also contains five stratified edges improving the score from $-12548.47$ for the ordinary graph to $-12532.40$.

\section{Discussion}
Gaussian graphical models have gained popularity for a multitude of applications involving analysis of multivariate continuous systems, in analogy with their counterparts for discrete variables. For discrete graphical models several approaches have been proposed for generalizing the dependence structure by local or context-specific independencies such that more flexible model classes are obtained which can reveal additional hidden independencies in data. To the best of our knowledge, such an approach has not been previously adopted for Gaussian graphical models. Using the concept of stratification of the joint outcome space, introduced by \citet{Nyman13} for discrete graphical models, we showed here that context-specific independence generalizes well to the multivariate Gaussian distribution, such that both interpretability and inferential tractability are preserved. Despite of our stratification of the outcome space, the resulting models are not typical mixture-type latent class models, for which inference is notoriously challenging, especially in the multivariate setting considered here. An interesting further generalization of the SGGM class would be to consider an adaptation to directed Gaussian graphical models, for which Bayesian learning has been recently considered in \cite{Consonni12}. A potential solution to obtaining such a generalization would be to employ the concept of labeled directed acyclic graphs, introduced for discrete-valued systems by \citet{Pensar13}. Valuable additional insight to the applicability of SGGMs could also be obtained by developing faster inference tools suitable for the analysis of large continuous systems.

\section*{Acknowledgement}
H.N. and J.P. were supported by the Foundation of \AA bo Akademi University, as part of the grant for the Center of Excellence in Optimization and Systems Engineering. J.C. was supported by the ERC grant no. 239784 and academy of Finland grant no. 251170.

\section*{Appendix A}
\begin{proof}[Proof of Theorem~\ref{theoremEF}]
It is a well established fact that multivariate normal distributions and distributions in GGMs belong to the exponential family. To show that the distribution in a decomposable SGGM belongs to the curved exponential family, we consider first a simple example and then the general case to provide clearer intuition for the reader. The density function of a multivariate normal distribution with zero mean can, using $x_i$ to denote the $i$th element of the column vector $x$, be written as
\[
f_{\Sigma}(x)=(2\pi)^{-d/2}|K|^{1/2}e^{-\frac{1}{2}x^T Kx}=\frac{e^{-\frac{1}{2}\sum_{i=1}^{d}x_{i}^{2}k_{ii}-\sum\sum_{i<j}x_{i}x_{j}k_{ij}}}{(2\pi)^{d/2}|K|^{-1/2}},
\]
which is in the exponential family form with $h(x)=(2\pi)^{-d/2}$, $\psi(\Sigma)=\log(|K|^{-1/2})$,
\[
\eta(\Sigma)=(k_{11},\ldots,k_{dd},k_{12},\ldots,k_{(d-1)d}),
\]
and
\[
T(x)=(-\frac{1}{2}x_{1}^{2},\ldots,-\frac{1}{2}x_{d}^{2},-x_{1}x_{2},\ldots,-x_{d-1}x_{d}).
\]
Consider now the SG depicted in Figure \ref{simpleSGGM}. Following from \eqref{SGGMpdf} a density function following the dependence structure defined by this SG factorizes as
\begin{equation}
g_{\Sigma}(x)=\frac{1}{Z}((2\pi)^{-d/2}|K^{(1)}|^{1/2}e^{-\frac{1}{2}x^T K^{(1)} x}I_{x_{1}\leq0}+(2\pi)^{-d/2}|K^{(2)}|^{1/2}e^{-\frac{1}{2}x^T K^{(2)} x}I_{x_{1}>0}).
\label{pdf1}
\end{equation}
The covariance matrices $\Sigma^{(1)}$ and $\Sigma^{(2)}$ generating the precision matrices $K^{(1)}$ and $K^{(2)}$, respectively, are identical to $\Sigma$ except for the element $\sigma_{2,3}^{(2)}$ (and $\sigma_{3,2}^{(2)}$) which is modified such that the corresponding value in the precision matrix $k_{2,3}^{(2)}$ equals zero. The equation in \eqref{pdf1} can be re-written as
\[
g_{\Sigma}(x)=\frac{1}{Z}\left(  \frac{e^{-\frac{1}{2}\sum_{i=1}^{d}x_{i}^{2}k_{ii}^{(1)}-\sum\sum_{i<j}x_{i}x_{j}k_{ij}^{(1)}}}{(2\pi)^{d/2}|K^{(1)}|^{-1/2}}I_{x_{1}\leq0}+\frac{e^{-\frac{1}{2}\sum_{i=1}^{d}x_{i}^{2}k_{ii}^{(2)} - \sum\sum_{i<j}x_{i}x_{j}k_{ij}^{(2)}}}{(2\pi)^{d/2}|K^{(2)}|^{-1/2}}I_{x_{1}>0}\right).
\]
Noting that $Z$ is determined by $\Sigma$ we can see that this density defines a distribution in the curved exponential family by setting $h(x)=(2\pi)^{-d/2}$, $\psi(\Sigma)=\log(Z)$,
\begin{align*}
\eta(\Sigma)=( &  k_{11}^{(1)},\ldots,k_{dd}^{(1)},k_{12}^{(1)},\ldots, k_{(d-1)d}^{(1)},\log(|K^{(1)}|^{1/2}),\\
&  k_{11}^{(2)},\ldots,k_{dd}^{(2)},k_{12}^{(2)},\ldots,k_{(d-1)d}^{(2)}, \log(|K^{(2)}|^{1/2})),
\end{align*}
and
\begin{align*}
T(x)=( &  -\frac{1}{2}f_{1}^{(1)}(x)^{2},\ldots,-\frac{1}{2}f_{d}^{(1)}(x)^{2},-f_{1}^{(1)}(x)f_{2}^{(1)}(x),\ldots,-f_{d-1}^{(1)}(x)f_{d}^{(1)}(x),I_{x_{1}\leq0}, \\
&  -\frac{1}{2}f_{1}^{(2)}(x)^{2},\ldots,-\frac{1}{2}f_{d}^{(2)}(x)^{2}, -f_{1}^{(2)}(x)f_{2}^{(2)}(x),\ldots,-f_{d-1}^{(2)}(x)f_{d}^{(2)}(x),I_{x_{1}>0}).
\end{align*}
The functions $f_{i}^{(1)}(x)$ and $f_{i}^{(2)}(x)$ are defined as
\newline
\begin{minipage}[t]{0.6\textwidth}
\[
f_i^{(1)}(x)=
\begin{cases}
x_i, & \text{if } x_1 \leq 0 \\
0, & \text{if } x_1 > 0
\end{cases}
\]
\end{minipage}\hspace{-2cm} \begin{minipage}[t]{0.4\textwidth}
\[
f_i^{(2)}(x)=
\begin{cases}
x_i, & \text{if } x_1 > 0 \\
0, & \text{if } x_1 \leq 0
\end{cases}
\]
\end{minipage}\newline The general case is proven following the same approach as in the example above. The density function is written as
\[
g_{\Sigma}(x)=\frac{1}{Z}\sum_{r=1}^{\rho}\frac{e^{-\frac{1}{2}\sum_{i=1}^{d}x_{i}^{2}k_{ii}^{(r)} - \sum\sum_{i<j}x_{i}x_{j}k_{ij}^{(r)}}}{(2\pi)^{d/2}|K^{(r)}|^{-1/2}}I_{c^{(r)}}(x),
\]
which again can be identified as a member of the curved exponential family by setting $h(x)=(2\pi)^{-d/2}$, $\psi(\Sigma)=\log(Z)$,
\begin{align*}
\eta(\Sigma)=( &  k_{11}^{(1)},\ldots,k_{dd}^{(1)},k_{12}^{(1)},\ldots, k_{(d-1)d}^{(1)},\log(|K^{(1)}|^{1/2}),\ldots,\\
&  k_{11}^{(\rho)},\ldots,k_{dd}^{(\rho)},k_{12}^{(\rho)},\ldots, k_{(d-1)d}^{(\rho)},\log(|K^{(\rho)}|^{1/2})),
\end{align*}
and
\begin{align*}
T(x)=( &  -\frac{1}{2}f_{1}^{(1)}(x)^{2},\ldots,-\frac{1}{2}f_{d}^{(1)}(x)^{2},-f_{1}^{(1)}(x)f_{2}^{(1)}(x),\ldots,-f_{d-1}^{(1)}(x)f_{d}^{(1)}(x),I_{c^{(1)}}(x),\ldots, \\
&  -\frac{1}{2}f_{1}^{(\rho)}(x)^{2},\ldots,-\frac{1}{2}f_{d}^{(\rho)}(x)^{2},-f_{1}^{(\rho)}(x)f_{2}^{(\rho)}(x),\ldots,-f_{d-1}^{(\rho)}(x)f_{d}^{(\rho)}(x),I_{c^{(\rho)}}(x)).
\end{align*}
The functions $f_{i}^{(r)}(x)$ are defined similarly to above as
\[
f_{i}^{(r)}(x)=
\begin{cases}
x_{i}, & \text{if $c^{(r)}$ is satisfied,}\\
0, & \text{otherwise.}
\end{cases}
\]
This establishes that the distribution in a decomposable SGGM belongs to the curved exponential family.
\end{proof}

\section*{Appendix B}
The following proposal function is used to generate a candidate graph $G_{L}^{\ast}$ given the current graph $G_{L}$. The function is composed of five operators one of which is randomly chosen at each iteration:
\begin{enumerate}
\item Add or delete a randomly chosen edge in the underlying graph $G$, while ensuring that this operation does not result in a non-decomposable $G$. This operation may violate the compatibility between $G$ and some of the strata in $L$. As a result some strata may need to be altered or removed. A stratum $\mathcal{L}_{\{\delta,\gamma\}}$ is removed if the edge $\{\delta, \gamma\}$ is included in a separator, if it is included in a clique containing less then three nodes, or if the edge is not present in the new underlying graph. Let $L^O_{\{\delta,\gamma\}}$ and $L^N_{\{\delta,\gamma\}}$ denote the set of nodes adjacent to both $\delta$ and $\gamma$ in the old and new underlying graph, respectively. The stratum $\mathcal{L}_{\{\delta,\gamma\}}$ is altered if $L^O_{\{\delta,\gamma\}} \neq L^N_{\{\delta,\gamma\}}$. If a node $\zeta$ is present in $L^O_{\{\delta,\gamma\}}$ but not present in $L^N_{\{\delta,\gamma\}}$ the conditions $a_{\zeta} < X_{\zeta} < b_{\zeta}$ are removed from $\mathcal{L}_{\{\delta,\gamma\}}$. Alternatively, if $\zeta$ is present in $L^N_{\{\delta,\gamma\}}$ but not in $L^O_{\{\delta,\gamma\}}$ the condition $-\infty < X_{\zeta} < \infty$ is added to $\mathcal{L}_{\{\delta,\gamma\}}$. Merging two cliques by adding an edge can lead to the negation of the property that all stratified edges in a clique have at least one node in common, which is an essential property of decomposable SGs. If this occurs the least possible amount of strata is removed in order to restore the property. Additionally, if an edge is added, a random stratum may be appended to the edge in accordance with the procedure in step 2.

\item Add a random stratum to a randomly chosen eligible edge in $G$. An edge is defined as eligible if adding a stratum to the edge does not result in a non-decomposable SG. This operation is performed by first randomly choosing an eligible edge ${\{\delta,\gamma\}}$. For each node $\zeta \in L_{\{\delta,\gamma\}}$ draw two values uniformly from the interval $(-2 \sigma_{\zeta}, 2 \sigma_{\zeta})$, where $\sigma_{\zeta}$ is the standard deviation of $X_{\zeta}$ calculated from the data $\textbf{X}_{\zeta}$. Use $a_{\zeta}$ and $b_{\zeta}$ to denote the smaller and larger value, respectively. If $a_{\zeta} < \min (\textbf{X}_{\zeta})$ set $a_{\zeta} = -\infty$ and analogously, if $b_{\zeta} > \max (\textbf{X}_{\zeta})$ set $b_{\zeta} = \infty$. Using the attained values the new stratum can be written in the form of equation (3) of the main manuscript. 

\item Remove a randomly chosen stratum from $L$.

\item Change the lower and upper limits of all strata associated to edges in a randomly chosen clique of $G$. This is done by randomly choosing a clique in $G$ and for each stratum associated to an edge belonging to this clique change the stratum's upper and lower limits. Let $a_{\zeta}$ and $b_{\zeta}$ be the lower and upper limits in such a stratum for the variable $X_{\zeta}$. If $a_{\zeta} = -\infty$ set $a_{\zeta} = -3 \sigma_{\zeta}$, draw a value $q$ uniformly from the interval $(-0.5 \sigma_{\zeta}, 0.5 \sigma_{\zeta})$ and set $a^*_{\zeta} = a_{\zeta} + q$, repeat this procedure for $b_{\zeta}$. If $a^*_{\zeta} > b^*_{\zeta}$ switch the values, insuring that $a^*_{\zeta} < b^*_{\zeta}$. If $a^*_{\zeta} < \min (\textbf{X}_{\zeta})$ set $a^*_{\zeta} = -\infty$ and analogously, if $b^*_{\zeta} > \max (\textbf{X}_{\zeta})$ set $b^*_{\zeta} = \infty$. Set the new lower and upper limits for $X_{\zeta}$ in the stratum to be $a^*_{\zeta}$ and $b^*_{\zeta}$, respectively.

\item Remove a randomly chosen stratum from $L$, then add a random stratum to a randomly chosen eligible edge following the procedure defined in step 2.
\end{enumerate}
The probabilities with which the operations are selected can be chosen quite freely. However, in order to guarantee that each state can be reached from any other state, operations 1-3 need to be selected with a strictly positive probability. The resulting candidate is transformed to its maximal regular counterpart, if no such counterpart exists a new candidate is generated. This transformation may lead to a change in the number of free parameters and subsequently a change in the score. In other cases the transformation can remove ambiguity concerning graphs with different appearance inducing identical dependence structures.

\section*{Appendix C}
The following conditional distributions, for variables $X_{1}$-$X_{7}$ given in Table \ref{tabCond}, are used to define a probability distribution following the dependence structure induced by the SG in Figure \ref{data8}. The stochastic variables $Y_{1}$-$Y_{9}$ follow the standard normal distribution.
\begin{table}[h]
\begin{center}
\begin{tabular}{clr}
\hline
Variable & Condition & Conditional Distribution \\
\hline
$X_{1}$ &  & $X_{1} \sim N(0,1)$ \\
$X_{2}$ &  & $X_{2} = (X_{1} + Y_{1}) / \sqrt{2}$ \\
$X_{3}$ &  & $X_{3} = (X_{1} + X_{2} + Y_{2}) / \sqrt{3}$ \\
$X_{4}$ & $X_{1} \in (0,\infty)$ $\wedge$ ($X_{2} \in(0,\infty)$ $\vee$ $X_{3} \in(-\infty,0)$) & $X_{4} = (X_{1} + Y_{3})/\sqrt{2}$\\
& $X_{1} \in(-\infty,0]$ $\vee$ ($X_{2} \in(-\infty,0]$ $\wedge$ $X_{3} \in[0,\infty)$) & $X_{4} = (X_{1} + X_{2} + X_{3} + Y_{4})/\sqrt{4}$\\
$X_{5}$ & $X_{2} \in(0,2)$ & $X_{5} = (X_{2} + Y_{5}) / \sqrt{2}$\\
& $X_{2} \in(-\infty,0]$ $\vee$ $X_{2} \in[2,\infty)$ & $X_{5} = (X_{1} + X_{2} +Y_{6})/\sqrt{3}$\\
$X_{6}$ &  & $X_{6} = (X_{5} + Y_{7})/\sqrt{2}$\\
$X_{7}$ & $X_{5} \in(-1,1)$ & $X_{7}=(X_{5} + Y_{8})/\sqrt{2}$\\
& $X_{5} \in(-\infty,-1]$ $\vee$ $X_{5} \in[1,\infty)$ & $X_{7}=(X_{5} + X_{6} + Y_{9})/\sqrt{3}$ \\
\hline
\end{tabular}
\caption{Conditional distributions used to specify a probability distribution following the dependence structure of the SG in Figure \ref{data8}.}
\label{tabCond}
\end{center}
\end{table}

\section*{Appendix D}
The following genes and corresponding indexes are used in the gene expression data example.
\begin{table}[h]
\begin{center}
\begin{tabular}{cccccc}
\hline
Index & Gene & Index & Gene & Index & Gene \\
\hline
1 & AS3MT & 6 & DCN & 11 & HS.387405 \\
2 & C10orf4 & 7 & DENND2D & 12 & HS.403212  \\
3 & C20orf144 & 8 & NSUN7 & 13 & HS.500666 \\
4 & CHAD & 9 & FLJ27255 & 14 & HS.520628  \\
5 & CSTL1 & 10 & FRMD5 & 15 & HS.537675 \\
\hline
\end{tabular}
\caption{Genes included in gene expression data example.}
\label{tabGene}
\end{center}
\end{table}

\bibliographystyle{henrik}
\bibliography{SGGMbib}

\begin{thebibliography}{32}
\providecommand{\natexlab}[1]{#1}
\expandafter\ifx\csname urlstyle\endcsname\relax
  \providecommand{\doi}[1]{doi:\discretionary{}{}{}#1}\else
  \providecommand{\doi}{doi:\discretionary{}{}{}\begingroup
  \urlstyle{rm}\Url}\fi

\bibitem[{Atay-Kayis and Massam(2005)}]{Atay05}
Atay-Kayis, A. and Massam, H.
\newblock A {M}onte {C}arlo method for computing the marginal likelihood in
  nondecomposable {G}aussian graphical models.
\newblock Biometrika, 92:317--335 (2005).

\bibitem[{Boutilier et~al.(1996)Boutilier, Friedman, Goldszmidt, and
  Koller}]{Boutilier96}
Boutilier, C., Friedman, N., Goldszmidt, M., and Koller, D.
\newblock Context-specific independence in {B}ayesian networks.
\newblock In Proceedings of the Twelfth Annual Conference on Uncertainty in
  Artificial Intelligence, pages 115--123 (1996).

\bibitem[{Carvalho and Scott(2009)}]{Carvalho09}
Carvalho, C.~M. and Scott, J.~G.
\newblock Objective {B}ayesian model selection in {G}aussian graphical models.
\newblock Biometrika, 96:497--512 (2009).

\bibitem[{Chickering et~al.(1997)Chickering, Heckerman, and
  Meek}]{Chickering97}
Chickering, D.~M., Heckerman, D., and Meek, C.
\newblock A {B}ayesian approach to learning {B}ayesian networks with local
  structure.
\newblock In Proceedings of the Thirteenth conference on Uncertainty in
  artificial intelligence, pages 80--89 (1997).

\bibitem[{Consonni and Rocca(2012)}]{Consonni12}
Consonni, G. and Rocca, L.~L.
\newblock Objective {B}ayes factors for {G}aussian directed acyclic graphical
  models.
\newblock Scand. J. Stat., 39:743--756 (2012).

\bibitem[{Corander(2003)}]{Corander03a}
Corander, J.
\newblock Labelled graphical models.
\newblock Scand. J. Stat., 30:493--508 (2003).

\bibitem[{Corander et~al.(2008)Corander, Ekdahl, and Koski}]{Corander08}
Corander, J., Ekdahl, M., and Koski, T.
\newblock Parallell interacting {MCMC} for learning of topologies of graphical
  models.
\newblock Data Min. Knowl. Discov., 17:431--456 (2008).

\bibitem[{Corander et~al.(2006)Corander, Gyllenberg, and Koski}]{Corander06}
Corander, J., Gyllenberg, M., and Koski, T.
\newblock {B}ayesian model learning based on a parallel {MCMC} strategy.
\newblock Stat. Comput., 16:355--362 (2006).

\bibitem[{DasGupta(2011)}]{DasGupta11}
DasGupta, A.
\newblock Probability for Statistics and Machine Learning.
\newblock Springer, New York (2011).

\bibitem[{Dawid and Lauritzen(1993)}]{Dawid93}
Dawid, A. and Lauritzen, S.
\newblock Hyper-{M}arkov laws in the statistical analysis of decomposable
  graphical models.
\newblock Ann. Statist., 21:1272--1317 (1993).

\bibitem[{Dempster(1972)}]{Dempster72}
Dempster, A.
\newblock Covariance selection.
\newblock Biometrics, 28:157--175 (1972).

\bibitem[{Edwards(2000)}]{Edwards00}
Edwards, D.
\newblock Introduction to Graphical Modelling.
\newblock Springer-Verlag, New York, 2nd edition (2000).

\bibitem[{Foygel and Drton(2010)}]{Foygel10}
Foygel, R. and Drton, M.
\newblock Extended bayesian information criteria for {G}aussian graphical
  models.
\newblock In Advances in Neural Information Processing Systems 23, pages
  604--612 (2010).

\bibitem[{Geiger and Heckerman(1996)}]{Geiger96}
Geiger, D. and Heckerman, D.
\newblock Knowledge representation and inference in similarity networks and
  {B}ayesian multinets.
\newblock Artificial Intelligence, 82:45--74 (1996).

\bibitem[{Giudici and Green(1999)}]{Giudici99}
Giudici, P. and Green, P.
\newblock Decomposable graphical {G}aussian model determination.
\newblock Biometrika, 86:785--801 (1999).

\bibitem[{Golumbic(2004)}]{Golumbic04}
Golumbic, M.~C.
\newblock Algorithmic graph theory and perfect graphs.
\newblock Elsevier, Amsterdam, 2nd edition (2004).

\bibitem[{Haughton(1988)}]{Haughton88}
Haughton, D.
\newblock On the choice of a model to fit data from an exponential family.
\newblock Ann. Statist., 16:342--355 (1988).

\bibitem[{Hiissa et~al.(2009)Hiissa, Elo, Huhtinen, Perheentupa, Poutanen, and
  Aittokallio}]{Hiissa09}
Hiissa, J., Elo, L.~L., Huhtinen, K., Perheentupa, A., Poutanen, M., and
  Aittokallio, T.
\newblock Resampling reveals sample-level differential expression in clinical
  genome-wide studies.
\newblock OMICS A Journal of Integrative Biology, 13:381--396 (2009).

\bibitem[{H{\o}jsgaard(2003)}]{Hojsgaard03}
H{\o}jsgaard, S.
\newblock Split models for contingency tables.
\newblock Comput. Statist. Data Anal., 42:621--645 (2003).

\bibitem[{H{\o}jsgaard(2004)}]{Hojsgaard04}
H{\o}jsgaard, S.
\newblock Statistical inference in context specific interaction models for
  contingency tables.
\newblock Scand. J. Stat., 31:143--158 (2004).

\bibitem[{Jones and West(2005)}]{Jones05}
Jones, B. and West, M.
\newblock Covariance decomposition in undirected {G}aussian graphical models.
\newblock Biometrika, 92:779--786 (2005).

\bibitem[{Kornblau et~al.(2009)Kornblau, Tibes, Qiu, Chen, Kantarjian,
  Andreeff, Coombes, and Mills}]{Kornblau09}
Kornblau, S.~M., Tibes, R., Qiu, Y.~H., Chen, W., Kantarjian, H.~M., Andreeff,
  M., Coombes, K.~R., and Mills, G.~B.
\newblock Functional proteomic profiling of {AML} predicts response and
  survival.
\newblock Blood, 113:154--164 (2009).

\bibitem[{Lauritzen(1996)}]{Lauritzen96}
Lauritzen, S.~L.
\newblock Graphical models.
\newblock Oxford University Press, Oxford (1996).

\bibitem[{Li and Gui(2006)}]{Li06}
Li, H. and Gui, J.
\newblock Gradient directed regularization for sparse {G}aussian concentration
  graphs with applications to inference of genetic networks.
\newblock Biostatistics, 7:302--317 (2006).

\bibitem[{Mardia et~al.(1979)Mardia, Kent, and Bibby}]{Mardia79}
Mardia, K.~V., Kent, J.~T., and Bibby, J.~M.
\newblock Multivariate Analysis.
\newblock Academic Press, London (1979).

\bibitem[{Nyman et~al.(2014)Nyman, Pensar, Koski, and Corander}]{Nyman13}
Nyman, H., Pensar, J., Koski, T., and Corander, J.
\newblock Stratified graphical models - context-specific independence in
  graphical models.
\newblock Bayesian Anal.: in press (2014).

\bibitem[{Pensar et~al.(2014)Pensar, Nyman, Koski, and Corander}]{Pensar13}
Pensar, J., Nyman, H., Koski, T., and Corander, J.
\newblock Labeled directed acyclic graphs: a generalization of context-specific
  independence in directed graphical models.
\newblock Data Min. Knowl. Discov.: in press (2014).

\bibitem[{Schwarz(1978)}]{Schwarz78}
Schwarz, G.
\newblock Estimating the dimension of a model.
\newblock Ann. Statist., 6:461--464 (1978).

\bibitem[{Sun and Li(2012)}]{Sun12}
Sun, H. and Li, H.
\newblock Robust {G}aussian graphical modeling via $l_1$ penalization.
\newblock Biometrics, 68:1197--1206 (2012).

\bibitem[{Whittaker(1990)}]{Whittaker90}
Whittaker, J.
\newblock Graphical models in applied multivariate statistics.
\newblock Wiley, Chichester (1990).

\bibitem[{Wong et~al.(2003)Wong, Carter, and Kohn}]{Wong03}
Wong, F., Carter, C.~K., and Kohn, R.
\newblock Efficient estimation of covariance selection models.
\newblock Biometrika, 90:809--830 (2003).

\bibitem[{Yuan and Lin(2007)}]{Yuan07}
Yuan, M. and Lin, Y.
\newblock Model selection and estimation in the {G}aussian graphical model.
\newblock Biometrika, 94:19--35 (2007).

\end{thebibliography}

\end{document}